\documentclass[12pt]{amsart}
\usepackage{amsmath}
\usepackage{amssymb}
\usepackage{fullpage}
\usepackage[pdfencoding=auto]{hyperref}
\usepackage{dsfont}
\usepackage{eucal}
\usepackage[all]{xy}
\SelectTips{eu}{}

\newcommand{\D}{\mathsf{D}}
\newcommand{\Db}{\D^{\mathsf{b}}}

\newcommand{\Dsg}{\D_{\mathsf{sg}}}
\newcommand{\Dcsg}{\D_{\mathsf{csg}}}
\newcommand{\End}{\mathsf{End}}
\newcommand{\iEnd}{\mathcal{E}\mathit{nd}}

\newcommand{\Ext}{\mathsf{Ext}}
\newcommand{\HH}{H\!H}
\newcommand{\Hom}{\mathsf{Hom}}
\newcommand{\iHom}{\mathcal{H}\mathit{om}}

\newcommand{\ind}{\mathsf{ind}\,}

\newcommand{\phat}{{}^{^\wedge}_p}
\newcommand{\Proj}{\mathsf{Proj}}

\newcommand{\Thick}{\mathsf{Thick}}

\newcommand{\bB}{\mathbb B}

\newcommand{\bZ}{\mathbb Z}

\newcommand{\cE}{\mathcal E}
\newcommand{\cH}{\mathcal H}

\newcommand{\cT}{\mathcal{T}}

\newcommand{\fa}{\mathfrak{a}}

\newcommand{\sfT}{\mathsf{T}}

\renewcommand{\le}{\leqslant}
\renewcommand{\ge}{\geqslant}

\numberwithin{equation}{section}  

\theoremstyle{plain}
\newtheorem{lemma}[equation]{Lemma}
\newtheorem{theorem}[equation]{Theorem}
\newtheorem{proposition}[equation]{Proposition}
\newtheorem{corollary}[equation]{Corollary}

\theoremstyle{definition}

\newtheorem{definition}[equation]{Definition}
\newtheorem{example}[equation]{Example}

\theoremstyle{remark} 
\newtheorem{remark}[equation]{Remark}

\author{Dave Benson}
\address{Institute of Mathematics, Fraser Noble Building, University of Aberdeen, Aberdeen
  AB24 3UE, United Kingdom}

\author{John Greenlees} 
\address{Mathematics Institute, Zeeman Building, University of
  Warwick, Coventry CV4 7AL, United Kingdom}

\title{The singularity and cosingularity categories of $C^*BG$ for
  groups with cyclic Sylow $p$-subgroups.}
\begin{document}

\begin{abstract}
We construct a differential graded algebra (DGA) modelling certain $A_\infty$ algebras
associated with a finite group $G$ with cyclic Sylow subgroups, namely
$H^*BG$ and $H_*\Omega BG\phat$. We use our
construction to investigate the singularity and cosingularity
categories of these algebras. We give a complete classification of the
indecomposables in these categories, and describe the
Auslander--Reiten quiver. The theory applies to Brauer tree algebras
in arbitrary characteristic, and we end with an example in
characteristic zero coming from the Hecke algebras of symmetric groups.
\end{abstract}

\maketitle

\section{Introduction}
Our purpose is to study the singularity categories of certain
$A_\infty$ algebras over a field $k$.  
We were led to these examples from the 
representation theory in characteristic $p$ of finite groups with 
cyclic Sylow $p$-subgroups, but our earlier
work~\cite{Benson/Greenlees:bg8}, on which we build, showed these examples were members
of a more general family of examples: the general case illuminates
those we first considered, and the other examples also occur
elsewhere. 

In fact the examples occur in Koszul dual pairs $A$ and $B$. The BGG
correspondence shows that it is illuminating to consider both members
of the pair together: the classical example occurs with $A$ an
exterior algebra and $B$ the Koszul dual polynomial algebra. The
singularity category of $A$ is equivalent to the cosingularity
category of $B$, which by a theorem of Serre is the  bounded derived
category of quasicoherent sheaves on $\Proj (B)$. Since $A$ is finite
dimensional, its cosingularity category is trivial; since $B$ is regular, its
singularity category is trivial. In this case both $A$ and $B$ are
formal as $k$-algebras.

We consider here a family of the next simplest cases consisting of a non-formal $A_\infty$ 
$k$-algebra $A$, usually with homology 
$$H_*(A) =k[\tau]\otimes \Lambda (\xi)$$
where $\tau$ has even degree $2b$ and $\xi$ has odd degree $2a-1$.
The family of examples we study is determined by $a, b$ and two
further parameters $h, \ell \ge 2$ related by $ah-b\ell=1$.
The parameter 
$h$ is the length of the shortest non-trivial Massey product 
(when $h=2$ the
homology ring is a little different to that above, since
$\xi^2=-\tau^\ell$).  We give a full description of $A$ in Section \ref{se:A}. 
It is shown in \cite{Greenlees/Stevenson:2020a} that the BGG 
correspondence extends to a more general $A_\infty$ context, and it is
again natural to consider the Koszul dual $B$. In fact $B$ is of
exactly the same form as $A$ but with different degrees, 
and the parameters $h$ and $\ell$
exchanged. In most cases it again has homology of  form
$$H_*(B) =\Lambda (t) \otimes k[x]$$
 where $x$ is of even degree $-2a$ and $t$ is of odd degree $-2b-1$. 
 The parameter 
 $\ell$ is the length of the shortest non-trivial Massey product
 (when $\ell=2$ the homology is a little different to that above, since
 $t^2=-x^h$).  We give a full description of $B$ in Section \ref{se:B}. 
In this case both $A$ and $B$ have singularity categories that are
 non-trivial and we are able to give a complete description: they each have finitely many
 indecomposable objects and we describe their Auslander--Reiten
 quivers. This behaviour is rather special. In general, even when the singularity
category of $A$ has finitely many indecomposables, the singularity
category of $B$ can have infinitely many. The behaviour of $A$ and $B$
is also quite different
to the behaviour of the formal algebras $H_*(A)$ and $H_*(B)$, whose
singularity and cosingularity categories all have infinitely many indecomposable objects.

Our first task  is to describe a small and explicit DG algebra $Q$ in the same quasi-isomorphism 
class as $A$, with some good properties that make it suitable for both 
theoretical and computational work.  As a step towards 
$Q$, we first introduce an auxiliary DG algebra $R$ in Section  
\ref{se:R}, which embodies the algebra of an odd element all of whose
Massey powers  vanish: it is generated by elements $\xi_1, \xi_2, \ldots $,
and has homology $H_*(R)=\Lambda (\xi)$.  The algebra $Q$ can be viewed as a
universal object for an algebra with an $h$-fold Massey power of an
element of odd degree which is an $\ell$th power of an element of even
degree. It is generated by elements $\tau$, $\xi_1,\dots,\xi_{h-1}$, with $\tau$ and $\xi_1$
representing elements $\tau$, $\xi$ in $H_*Q\cong A$. Explicit
formulas for the relations and the action of the differential $d$ are given in
Section~\ref{se:Q}. The element $\tau$ of $Q$ is central, so
$Q$ may be regarded as an algebra over $k[\tau]$. 
Our principal goal is to determine the structure of the singularity
and cosingularity categories of $A$ and $B$ (see Section~\ref{sec:Db}
for definitions). Our main theorems classify the indecomposable
objects in these categories, 
see Theorems~\ref{th:DcsgA} and \ref{th:classification}.

\begin{theorem}\label{th:main}
The equivalences of triangulated categories 
$\Db(A) \simeq\Db(B)$ induce equivalences 
\[ 
\Dcsg(A) \simeq \Db(A[\tau^{-1}]) \simeq
  \Dsg(B). \]
The latter categories satisfy the Krull--Schmidt theorem, and have
$|b|(h-1)$ isomorphism classes of indecomposable objects, in
$[h/2]$ orbits of the shift functor $\Sigma$. The Auslander--Reiten
quiver is isomorphic to $\bZ A_{h-1}/T^{|b|}$, where $T$ is the
translation functor $\Sigma^{-2a}$. This is a cylinder of height $h-1$
and circumference $|b|$.  The functor $\Sigma$ switches the two ends of
the cylinder.
\end{theorem}

We give explicit descriptions of the indecomposable objects, both as
elements of $\Dcsg(A)\simeq\Db(A[\tau^{-1}])$ and as elements of $\Dsg(B)$. Reversing
the roles of $A$ and $B$, swapping $h$ and $\ell$, and replacing $a$
by $-b$ and $b$ by $-a$ gives us the structure of
$\Dsg(A)\simeq\Dcsg(B)$, with $|a|(\ell-1)$ isomorphism classes of
indecomposable objects, coming in $[\ell/2]$ orbits of $\Sigma$.
We also give an explicit description of the Auslander--Reiten quivers
of these categories, and explain their position in the Amiot's
classification of finite triangulated categories. This leads to some
explicit models such as $\Dsg(B)\simeq \Db(A_{h-1})/\sfT^{|b|}$ in
Section~\ref{se:models}. 

Of course the simplicity of the $A_\infty$ structure is essential for
explicit calculations, but key structural features making a complete analysis
possible are the $\tau$-periodicity and the Tate duality of Theorem
\ref{th:Bduality}. The key ingredient for Tate duality is the fact that finitely
generated modules are automatically dualizable, which we proved in
this case using a Hochschild cohomology calculation. \medskip

We are especially interested in the following occurrence of the $A_\infty$
algebras $A$ and $B$. Let $p$ be an odd prime, 
let $G$ be a finite group with cyclic Sylow $p$-subgroups $P$ of order
$p^n$ and inertial index $q>1$, and let $k$ be a field of characteristic
$p$.  Omitting notation for coefficients in the field $k$, and writing 
$\Omega BG\phat$ for  the loop space on the Bousfield--Kan
mod $p$ completion of $BG$, we showed in \cite{Benson/Greenlees:bg8} 
that  the $A_\infty$ algebra structures on $A=H_*\Omega 
BG\phat$ and $B=H^*BG$
were of the above form with  $a=q$, $b=q-1$,  $h=p^n-(p^n-1)/q$, $\ell=p^n$.
 Then the DG algebra $Q$ describes a model for the DG algebra $C_*\Omega BG\phat$ up to
quasi-isomorphism, and by reversing the roles of $A$ and $B$ we obtain a model for $C^*BG$.

\begin{theorem}\label{th:main}
Let $h=p^n-(p^n-1)/q$, $\ell=p^n$. Then the equivalence of
triangulated categories $\Db(C^*BG)\simeq \Db(C_*\Omega BG\phat)$
induces equivalences
\[ \Dcsg(C_*\Omega BG\phat)\simeq\Db(C_*\Omega
  BG\phat)[\tau^{-1}]\simeq \Dsg(C^*BG). \] 
This is a finite Krull--Schmidt triangulated category with
$(q-1)(h-1)$ indecomposable objects in
$[h/2]$ orbits of the shift functor. It also induces equivalences
\[ \Dcsg(C^*BG)\simeq\Db(C^*BG)[x^{-1}]\simeq\Dsg(C_*\Omega
  BG\phat). \] 
This is a finite Krull--Schmidt
triangulated category with $q(p^n-1)$ indecomposable objects in
$[p^n/2]$ orbits of the shift functor.
\end{theorem}

Other examples of the $A_\infty$ algebras $A$ and $B$ 
occur every time an algebra is described by a Brauer
tree of finite representation type. These occur throughout 
representation theory, both in characteristic zero and in prime 
characteristic. For the sake of describing an example in 
characteristic zero, we discuss the Hecke algebras of 
symmetric groups. Let
$\cH=\cH(n,q)$ be the Hecke algebra of the symmetric group of degree
$n$ over a field $k$ of characteristic zero, where $q$ is a primitive
$\ell$th root of unity with $n=\ell>2$. Then letting $A$ be the
principal block of $\cH$ and $B$ be $\Ext^*_\cH(k,k)$, we obtain an
example with $a=n-1$, $b=n-2$, $h=n-1$ and $\ell=n$. We spell out the
consequences of our main theorem in this case, in Theorem~\ref{th:Hecke}.

\subsection*{Acknowledgements}
The authors are grateful to  EPSRC: the second author is supported by  
EP/P031080/1, which also enabled the first author to visit Warwick. 
The authors are also grateful to Bernhard Keller and Greg Stevenson for 
conversations related to this work.

The authors would also like to thank the Isaac Newton Institute for
Mathematical Sciences, Cambridge, for providing an opportunity to work
on this project during the simultaneous programmes \emph{`$K$-theory, algebraic cycles and motivic homotopy
  theory'} and \emph{`Groups, representations and applications: new
  perspectives'} (one author was supported by each programme). 

\section{The DG Hopf algebra $R$}\label{se:R}

We begin by looking at the DG Hopf algebra $R$ over a field $k$.
As a graded algebra over $k$, $R$ is free with odd degree generators
$\xi_1,\xi_2,\xi_3,\dots$, and the differential is given by
\[ d(\xi_i)=\sum_{j+k=i}\xi_j\xi_k\qquad (i\ge 1). \]
Thus $d(\xi_1)=0$, $d(\xi_2)=\xi_1^2$,
$d(\xi_3)=\xi_1\xi_2+\xi_2\xi_1$, and so on. 

\begin{remark} 
To motivate this, we factor out the differential ideal
  of $R$ generated by $\xi_i$ for $i\ge h+1$ and 
take the DG-subalgebra $R_h$ generated by  $\xi_i$ for $i\le
h-1$. The element $\mu_h=d(\xi_h)$ lies in $R_h$ and 
represents the $h$-fold Massey power of the homology class of $\xi_1$
up to sign. Thus a DGA map from $\theta\colon R_h\to C$ in which $\theta
(\xi_1)=c$
 shows that the $h$-fold Massey power of $[c]$ is
defined and gives an element $[\theta (\mu_h)]\in \pm\langle [c], \cdots,
[c]\rangle$. Accordingly, in $R$ itself, all Massey powers of $[\xi_1]$ contain zero. 
\end{remark}

The antipode $S$ on $R$ is the anti-automorphism of algebras
given on generators by
$S(\xi_i)=-\xi_i$. The comultiplication $\Delta\colon R \to
R\otimes R$ is defined on generators by
\[ \Delta(\xi_i)=\xi_i\otimes 1 + 1\otimes \xi_i. \]

Note that if $\xi_1$ has degree $2a-1$ then $\xi_i$ has degree
$2ia-1$. So $R$ is either connected or coconnected, according to
whether $2a-1$ is positive or negative. We shall assume that $a\ne 0$,
so that $|\xi_1|\ne -1$,
which implies that each graded piece is finite dimensional.

\begin{lemma}\label{le:Rd2=0}
In $R$ we have $d^2=0$.
\end{lemma}
\begin{proof}
To show that $d^2=0$, we note that $dd(\xi_i)$ has two terms for
each way of writing $i$ as a sum of three positive integers. They
have opposite signs, so they cancel.
\end{proof}

\begin{lemma}\label{le:RDelta}
The map $\Delta\colon R \to R\otimes R$ is a map of DG algebras.
\end{lemma}
\begin{proof}
As an algebra, $R$ is free, so specifying the map on generators gives
a well defined map of algebras. We must check that it commutes with
the differential. Since the $\xi_j$ have odd degree and are primitive,
$\xi_j^2$ and $\xi_j\xi_k+\xi_k\xi_j$ are also primitive. So
$d(\xi_i)=\sum_{j+k=i}\xi_j\xi_k$ is also primitive, and hence $d\Delta(\xi_i)=\Delta
d(\xi_i)$.
\end{proof}

\begin{proposition}\label{pr:RHopf}
The definitions above make $R$ into a cocommutative DG Hopf algebra.
\end{proposition}
\begin{proof}
Lemmas~\ref{le:Rd2=0} and~\ref{le:RDelta} show that $R$ is a DG bialgebra. It is easy to check
that the antipode satisfies the identity
$S(x_{(1)})x_{(2)}=x_{(1)}S(x_{(2)})=0$ in Sweedler notation, for elements of non-zero
degree; this only needs checking on the generators, where it is
clear. Cocommutativity also only needs checking on generators.
\end{proof}

\begin{lemma}\label{le:H*R}
$H_*R=\Lambda(\xi_1)$.
\end{lemma}
\begin{proof}
Define a linear map $\delta\colon R \to R$ sending a monomial 
of the form $\xi_1\xi_i f$ to $\xi_{i+1}f$,
and sending all other monomials to zero. Then we have
\begin{align*}
\delta d(\xi_1\xi_i f) & = \delta(-\xi_1(\xi_1\xi_{i-1}+\cdots+\xi_{i-1}\xi_1)f+\xi_1\xi_idf)\\
&= -(\xi_2\xi_{i-1}+\cdots + \xi_i\xi_1)f+\xi_{i+1}df \\
d\delta(\xi_1\xi_i f) &= d(\xi_{i+1} f) = (\xi_1\xi_i+\cdots + \xi_i\xi_1)f- \xi_{i+1}df \\
(\delta d + d\delta)(\xi_1\xi_i f) &= \xi_1\xi_i f,
\end{align*}
while for $j>1$ we have
\begin{align*}
d\delta(\xi_jf)&=d(0)=0 \\
\delta  d(\xi_jf)&=\delta((\xi_1\xi_{j-1}+\dots+\xi_{j-1}\xi_1)f-\xi_jdf) =\xi_jf \\
(d\delta + \delta d)(\xi_jf)&=\xi_jf.
\end{align*}
Thus $\delta d + d \delta$ is the identity
on all monomials apart from $1$ and $\xi_1$, on which it vanishes. So 
$\delta$ defines a homotopy from the identity map of $R$ 
to the projection onto the linear span of $1$ and $\xi_1$. 
It follows that $H_* R = \Lambda(\xi_1)$.
\end{proof}

\begin{definition}\label{def:wt}
The \emph{weight} of a monomial in $R$ is the sum of the subscripts
(and zero for constants).
The \emph{height} of a monomial in $R$ is the number of generators
$\xi_i$ that have to be multiplied to give the monomial (we might call
it degree, if that didn't already have a different meaning). Multiplication
in $R$ adds weights, and adds heights.
If $f(\xi_1,\dots,\xi_n)$ 
is an element of $R$, we write $f_{i,j}(\xi_1,\dots,\xi_n)$ for the sum of
the terms of $f$ with weight $i$ and height $j$. Thus $f=\sum_{i,j}f_{i,j}$.
\end{definition}

The differential $d$ preserves weight, and increases height by one,
so that $d(f_{i,j})=(df)_{i,j+1}$. Thus if $df=0$ then each $d(f_{i,j})=0$.

\section{The DG Hopf algebra $Q$}\label{se:Q}

In this section, we  let $h$ and $\ell$ be integers $\ge 2$ 
and describe the DG Hopf algebra $Q=Q_{h, \ell}$. In terms of the previous
section, the idea is that $Q_{h, \ell}$ is obtained from $R_h$ by adjoining an $\ell$th root to the element
$-\mu_h$. Thus a DGA map $\theta \colon Q_{h, \ell}\to C$ shows that the
$h$-fold Massey power of  $c=\theta (\xi_1)$ is defined and contains an
$\ell$th power of the adjoined variable.
 
The generators for $Q$ are $\xi_1,\dots,\xi_{h-1}$ in odd degree and 
$\tau$ in even degree. 
The relations and differential are as follows:
\begin{align*}
\tau\xi_i &= \xi_i \tau \qquad\ \ 1\le i \le h-1 \\
d(\tau)&= 0 \\
\sum_{j+k=i} \xi_j\xi_k &= \begin{cases}
d(\xi_i) & 1\le i \le h-1 \\
-\tau^\ell & i=h \\
0 & h+1 \le i \le 2h-2.
\end{cases}
\end{align*}
The antipode is the algebra anti-automorphism given by
$S(\xi_i)=-\xi_i$, $S(\tau)=-\tau$, and the comultiplication is given by 
\[ \Delta(\xi_i) = \xi_i \otimes 1 + 1\otimes \xi_i, \qquad
  \Delta(\tau)=\tau\otimes 1 + 1 \otimes \tau. \]

We write $|\xi_1|=2a-1$, and we assume that $a\ne 0$.
The relations imply that $|\xi_i|=2ia-1$
and $|\tau^\ell|=2ah-2$. So writing $2b$ for $|\tau|$ we have
$2b\ell = 2ah-2$, or equivalently 
\[ ah-b\ell=1. \] 
In particular, $a$ and $b$ are coprime, as are $h$ and $\ell$.

As well as this homological grading, 
we give $Q$ a second, internal grading by setting
\[ |\xi_i|=(2ia-1,i\ell), \qquad |\tau|=(2b,h). \] 
It is easy to check that
the relations, differential, and Hopf structure above respect this second grading.

\begin{example}\label{eg:h=2}
If $h=2$, the algebra $Q$ is generated by $\xi_1$ and $\tau$ 
with relations
\begin{align*}
d(\tau)&=0&\tau\xi_1&=\xi_1\tau\qquad\phantom{1\le i\le 1} \\
d(\xi_1)&=0&\xi_1^2&=-\tau^\ell
\end{align*}
In this case the differential is zero, and $Q$ is just the graded
algebra $k[\tau,\xi_1]/(\xi_1^2+\tau^\ell)$.
\end{example}

\begin{example}\label{eg:h=3}
If $h=3$, the algebra $Q$ is generated by $\xi_1$, $\xi_2$ and $\tau$
with relations
\begin{align*}
d(\tau)&=0 &\tau\xi_i&=\xi_i\tau\qquad 1\le i\le 2 \\
d(\xi_1)&=0&\xi_1\xi_2+\xi_2\xi_1&=-\tau^\ell \\
d(\xi_2)&=\xi_1^2&\xi_2^2&=0.
\end{align*}
\end{example}

\begin{example}\label{eg:h=4}
If $h=4$, the algebra $Q$ is generated by $\xi_1$, $\xi_2$, $\xi_3$ and 
$\tau$ with relations
\begin{align*}
d(\tau)&=0 &\tau\xi_i&=\xi_i\tau\qquad 1\le i\le 3 \\
d(\xi_1)&=0 & \xi_1\xi_3+\xi_2^2+\xi_3\xi_1&=-\tau^\ell\\ 
d(\xi_2)&=\xi_1^2 & \xi_2\xi_3+\xi_3\xi_2 &=0\\
d(\xi_3)&=\xi_1\xi_2+\xi_2\xi_1 &  \xi_3^2&=0. 
\end{align*}
\end{example}

\begin{lemma}
In the algebra $Q$, every element has a unique expression of the form
\[ f(\xi_1,\dots,\xi_{h-2})+\xi_{h-1}g(\xi_1,\dots,\xi_{h-2}) \] 
with coefficients in $k[\tau]$.
\end{lemma}
\begin{proof}
The algebra relations (ignoring the differential) can be rewritten
in the form 
\[ \xi_i\xi_{h-1}=\xi_{h-1}\phi_i(\xi_1,\dots,\xi_{h-2}), \]
with $1\le i \le h-1$ (note that $\phi_{h-1}=0$). 
Thus all occurrences of $\xi_{h-1}$ may
be moved to the beginning, and $\xi_{h-1}^2=0$. There are no relations among $\xi_1,\dots,\xi_{h-2}$.
\end{proof}

\begin{definition}\label{def:standard}
We shall refer to a monomial in $\xi_1,\dots,\xi_{h-2}$, or
$\xi_{h-1}$ times such a monomial, as a \emph{standard monomial} in
the variables $\xi_1,\dots,\xi_{h-1}$. The lemma shows that the
standard monomials form a basis for $Q$ as a free module over $k[\tau]$.
\end{definition}

\begin{lemma}\label{le:Qd2=0}
In the algebra $Q$, we have $d^2=0$.
\end{lemma}
\begin{proof}
The differential is given by
\[ d(f+\xi_{h-1}g)=
(df+(\xi_1\xi_{h-2}+\dots+\xi_{h-2}\xi_1)g)-\xi_{h-1}dg. \]
On the free subalgebra generated by $\xi_1,\dots,\xi_{h-2}$, the differential
is the same as in the algebra $R$ above, and so by 
Lemma~\ref{le:Rd2=0} we have $d^2=0$ on this subalgebra.
Thus we have
\begin{align*} 
d^2(f+\xi_{h-1}g)
&=d(df+(\xi_1\xi_{h-2}+\dots+\xi_{h-2}\xi_1)g-\xi_{h-1}dg)\\
&=d^2f + (\xi_1\xi_{h-2}+\dots+\xi_{h-2}\xi_1)dg
-(\xi_1\xi_{h-2}+\dots+\xi_{h-2}\xi_1)dg\\
&=0.
\qedhere
\end{align*}
\end{proof}

\begin{lemma}\label{le:QDelta}
The map $\Delta\colon Q \to Q \otimes Q$ is a map of DG algebras.
\end{lemma}
\begin{proof}
As in Lemma~\ref{le:RDelta}, for each $i>0$, $\sum_{j+k=i}\xi_j\xi_k$
is primitive, so both the relations and the elements $d(\xi_i)$ are
primitive.
Hence $\Delta$ takes relations to relations, and $d\Delta=\Delta d$.
\end{proof}

\begin{proposition}
The definitions above make $Q$ into a cocommutative DG Hopf algebra.
\end{proposition}
\begin{proof}
The proof of this is similar to the proof of
Proposition~\ref{pr:RHopf}, but using Lemmas~\ref{le:Qd2=0}
and~\ref{le:QDelta} in place of Lemmas~\ref{le:Rd2=0} and~\ref{le:RDelta}.
\end{proof}

\section{\texorpdfstring{The $A_\infty$ algebras $A$ and $B$}
{The A∞ algebras A and B}}\label{se:A}

In this section we introduce an $A_\infty$ algebra $A$ which is
quasi-isomorphic to the DG algebra $Q$ of the last section. We then
describe the Koszul dual $B$ of $A$. These are the $A_\infty$ algebras
discussed in our previous paper~\cite{Benson/Greenlees:bg8}.

Let $h$ and $\ell$ be positive integers with $h\ge 3$, and let $k$ be
a field. We are interested in the following $A_\infty$ algebra $A=A_{h,\ell}$.
The differential $m_1$ is zero, $m_2$ is the strictly associative
multiplication giving $A$ the ring structure of $k[\tau] \otimes
\Lambda(\xi)$, where $|\xi|=(2a-1,\ell)$ and $|\tau|=(2b,h)$, 
with $ah-b\ell=1$. We have
\begin{equation} \label{eq:mh}
m_h(\xi,\dots,\xi)=(-1)^{h(h-1)/2}\tau^\ell, 
\end{equation}
which implies 
\[  m_h(\tau^{j_1}\xi,\dots,\tau^{j_h}\xi)
=(-1)^{h(h-1)/2}\tau^{\ell+j_1+\dots+j_h} \]
for all $j_1,\dots,j_h\ge 0$. All $m_i$ for $i>2$ on all
other $i$-tuples of monomials give zero. We allow the
elements $\tau$ and $\xi$ to be either in positive or
in negative degree, and we grade everything homologically.
The relation \eqref{eq:mh} may be interpreted as saying that the
Massey product of $h$ copies of $\xi$ is equal to $-\tau^\ell$, 
the sign being the standard one relating Massey products with
$A_\infty$ structure;
see~\cite[Theorem~3.1]{Lu/Palmieri/Wu/Zhang:2009a},
\cite[Theorem~3.2]{Buijs/MorenoFernandez/Murillo:2020a}.

We extend the definition to $h=2$ by letting $A$ be the formal
$A_\infty$ algebra $k[\tau,\xi]/(\xi^2+\tau^\ell)$, with
$|\xi|=(2a-1,\ell)$, $|\tau|=(2b,2)$, $2a-b\ell=1$, $a\ne 0$, and with all $m_i$
apart from $m_2$ equal to zero.

We next show that the DG algebra $Q$ (forgetting the Hopf structure)
is quasi-isomorphic to $A$ as an $A_\infty$-algebra. In other words,
$A\cong H_*(Q)$, with
the $A_\infty$ structure given by Kadeishvili's theorem~\cite{Kadeishvili:1982a}.

\begin{theorem}\label{th:QsimA}
There is a quasi-isomorphism from the DG algebra $Q$ to the $A_\infty$
algebra $A$,  sending $\tau$ to $\tau$ and $\xi_1$ to $\xi$.
\end{theorem}
\begin{proof}
First, we show that $H_*Q$ is isomorphic to $A$ as an algebra over $k[\tau]$.
The proof is similar to the proof of Lemma~\ref{le:H*R},
but working over $k[\tau]$ instead of $k$.
Namely, we define a linear map $\delta\colon Q \to Q$ sending 
a monomial of the form $\xi_1\xi_i f$ to $\xi_{i+1}f$ for
$1\le i \le h-2$, and all other standard monomials 
(see Definition~\ref{def:standard}) to zero. Thus
$\delta(f+\xi_{h-1}g) = \delta(f)$. The same computation
as in Lemma~\ref{le:H*R} shows that $\delta d + d \delta$
is the identity on all monomials except those in
$k[\tau] + \xi_1 k[\tau]$, where it is zero. Thus $\delta$ defines
a homotopy from the identity map of $Q$ to the projection 
onto $k[\tau] + \xi_1 k[\tau]$. It follows that $H_*Q$ is
isomorphic to $A$ as a ring, with $\tau$ and $\xi_1$
corresponding to $\tau$ and $\xi$. The maps $m_i$ on $H_*Q$ are
easy to calculate using the elements $\xi_i$, and give zero for $i>2$ except in the case
of $m_h$, where it gives
$m_h(\xi,\dots,\xi)=(-1)^{h(h-1)/2}\tau^\ell$.
Using Kadeishvili's theorem~\cite{Kadeishvili:1982a} completes the proof.
\end{proof}

\begin{corollary}\label{co:HQ}
We have $H_*(Q) \cong \begin{cases} k[\tau] \otimes \Lambda(\xi) & h>2
  \\ k[\tau,\xi]/(\xi^2+\tau^\ell) & h=2.
\end{cases}$\qed
\end{corollary}

Let $B$ be the $A_\infty$ algebra whose algebra structure 
is $k[x] \otimes \Lambda(t)$ with $|x|=(-2a,-\ell)$,
$|t|=(-2b-1,-h)$, 
\[ m_\ell(x^{j_1}t,\dots,x^{j_\ell}t)
= (-1)^{\ell(\ell-1)/2}x^{h+j_1+\dots+j_\ell}, \]
and all $m_i$ with $i>2$ are zero on all other monomials.
In the exceptional case where $\ell=2$, we define $B$ to be the formal
$A_\infty$ algebra $k[x,t]/(t^2+x^h)$.
It was shown in~\cite{Benson/Greenlees:bg8} that $A$ and $B$ are
Koszul dual. Thus $A$ is quasi-isomorphic to $\iEnd_{\Db(B)}(k)$ and
$B$ is quasi-isomorphic to $\iEnd_{\Db(A)}(k)$. Here, $\iEnd_{\Db(B)}(k)$ denotes the
$A_\infty$ endomorphism ring whose homology is
\[ H_*\iEnd_{\Db(B)}(k)\cong\End_{\Db(B)}(k), \] 
and so on.

\section{Hochschild cohomology}\label{se:HH}

We will recall the definition of the Hochschild cohomology of an
$A_\infty$-algebra and then calculate it for the $A_\infty$ algebras 
$A$ and $B$ described in the previous section. The point
of this is that nilpotent elements in (for example) $\HH^*(A)$ control
certain uniform processes of construction in the category of
$A$-modules: we will make essential use of this in our proof of 
Theorem \ref{th:ThickAtau-1}.

The bar resolution $\bB(\fa)=\bigoplus_{n\ge 0}\fa^{\otimes (n+2)}$  
of an $A_\infty$ algebra $\fa$ is described in 
Section~3 of Getzler and Jones~\cite{Getzler/Jones:1990a},
see also Definition~12.6 of Stasheff~\cite{Stasheff:1970a}.
The action of the differential on $\fa^{\otimes (n+2)}$ in bar notation is
\begin{align*} 
d(x \otimes& [a_1|\dots|a_n]\otimes y) =
\sum_{j=0}^n \pm m_{j+1}(x,a_1,\dots,a_j) \otimes 
[a_{j+1}|\dots|a_n] \otimes y \\
&{} + \sum_{0\le  i+j \le n}\pm
x \otimes [a_1|\dots|a_i|m_j(a_{i+1},\dots,a_{i+j})|a_{i+j+1}|
\dots|a_n] \otimes y \\
&{} + \sum_{j=0}^n \pm x \otimes [a_1|\dots|a_{n-j}]
\otimes m_{j+1}(a_{n-j+1},\dots,a_n,y),
\end{align*}
where the signs are determined by the usual sign conventions.
Taking $\fa$-$\fa$-bimodule homomorphisms to a bimodule $M$, 
we obtain the differential on Hochschild cochains
\[ \Hom_{\fa,\fa}(\fa^{\otimes (n+2)},M) \cong 
\Hom_k(\fa^{\otimes n},M) \]
as follows:
\[ (df)[a_1|\dots|a_n]=d(f[a_1|\dots|a_n]) + \sum_{0\le i + j \le n}\pm
f[a_1|\dots|a_i|m_j(a_{i+1},\dots,a_{i+j})|a_{i+j+1}|\dots|a_n], \]
see also Section~1 
of Roitzheim and Whitehouse~\cite{Roitzheim/Whitehouse:2011a}.
The cohomology of this complex is $\HH^*(\fa,M)$. If $M=\fa$,
we write $\HH^*(\fa)$ for $\HH^*(\fa,\fa)$.

We filter $\bB(\fa)$ by number of bars, 
$F_i\,\bB(\fa)=\bigoplus_{n\le i}\fa^{\otimes (n+2)}$. 
This gives a filtration on Hochschild cochains, for which $F_i$ is
the cochains which vanish on $F_i\,\bB(\fa)$.
With this filtration, $F_0$ is the whole complex and $\bigcap_i F_i = 0$. This 
leads to a spectral sequence
in which the differentials $d_n$ are given by the terms involving $\pm m_{n+1}$. Thus the
$E^1$ page is the Hochschild complex of $H_*\fa$ with coefficients in
$H_*M$, and the $E^2$ page
is $\HH^*(H_*\fa,H_*M)$. So the spectral sequence takes
the form
\begin{equation}\label{eq:HHss} 
\HH^*(H_*\fa,H_*M) \Rightarrow \HH^*(\fa,M). 
\end{equation}
We are numbering everything homologically, so the Hochschild degrees
in $\HH^*H_*\fa$ are negative, and the spectral sequence lives in the
second and third quadrants.

Applying $\displaystyle\lim_\leftarrow$ with respect to $i$ to the exact sequences
$0 \to F_i \to F_0 \to F_0/F_i \to 0$, we get
\[ 0 = \bigcap_i F_i \to F_0 \to \lim_\leftarrow F_0/F_i \to
  {\lim_\leftarrow}^1 F_i \to 0. \]
So the spectral sequence is conditionally convergent, and is strongly
convergent if and only if $\displaystyle{\lim_\leftarrow}^1F_i=0$, 
which is equivalent to $\displaystyle{\lim_\leftarrow}^1E^i=0$ in the
spectral sequence, see for example Theorem~7.1 of 
Boardman~\cite{Boardman:1999a}. In particular, if each
graded piece of $E^r$ is finite dimensional for some $r$, then the spectral
sequence~\eqref{eq:HHss} is strongly convergent.

\begin{theorem}\label{th:Boardman}
Given a map of exact couples $(D,E)\to(D',E')$, where both spectral sequences are
conditionally convergent and live in the (homologically indexed)
left half plane, if the map $E^r \to E'^{\,r}$ is an isomorphism for some
$r$ then $D \to D'$ is an isomorphism of filtered graded groups.
\end{theorem}
\begin{proof}
This follows from Theorem~7.2 of Boardman~\cite{Boardman:1999a},
since the hypotheses imply that $E^\infty \to E'^{\,\infty}$ and 
$\displaystyle{\lim_\leftarrow}^1E^i\to
\displaystyle{\lim_\leftarrow}^1E'^{\,i}$
are isomorphisms. See also Theorem~B.7 of 
Greenlees and May~\cite{Greenlees/May:1995a}.
\end{proof}

We need to know that $A, B$ and $Q$ have the same Hochschild 
cohomology. Since our equivalences are slightly indirect we need some 
machinery to see the isomorphism preserves structure. Keller's article
\cite{Keller:dih} recalls the definition of $B_\infty$-algebras. 

\begin{proposition}\label{pr:Keller}
There are isomorphisms in the homotopy category of $B_\infty$ algebras between the 
Hochschild complexes of $Q$, $A$ and $B$.
\end{proposition}
\begin{proof}
For the algebras $Q$ and $A$, we apply the main theorem in Section~3.2
of~\cite{Keller:dih} to the quasi-isomorphism $Q\to A$ of 
Theorem~\ref{th:QsimA} to obtain an equivalence of Hochschild
complexes in the homotopy category of $B_\infty$ algebras.

Let $C$ be
the cobar construction on $Q$. This is an augmented DG coalgebra,
which is finite dimensional in each degree. As in Section~2 of
Keller~\cite{Keller:2021a}, this is the Koszul--Moore dual of $Q$.
Its graded $k$-linear dual
$C^*$ is ``the'' Koszul dual of $Q$, and is quasi-isomorphic to $B$.
Using the fact that $C$ is finite dimensional in each degree,
the Hochschild complex for $C$ is isomorphic to that for $C^*$ as a
$B_\infty$ algebra. Applying Theorem~3.3 of~\cite{Keller:2021a},
the Hochschild complexes of $Q$ and $C$ are equivalent in the homotopy 
category of $B_\infty$ algebras.
\end{proof}

Certainly an isomorphism in the homotopy category of $B_\infty$ algebras
induces an isomorphism of cohomology rings.

\begin{corollary}
We have $\HH^*Q \cong \HH^*A \cong \HH^*B$.\qed
\end{corollary}

\begin{theorem}\label{th:HHB}
Suppose that $h>2$ and $\ell>2$.
In the spectral sequence $\HH^*H_*B \Rightarrow \HH^*B$ the $E^2$ page
is given by 
\[ \HH^*H_*B \cong H_*A \otimes H_*B \cong k[x,\tau]\otimes
  \Lambda(t,\xi) \] 
where $|x|=(0,-2a,-\ell)$, $|t|=(0,-2b-1,-h)$,
$|\xi|=(-1,2a,\ell)$, and $|\tau|=(-1,2b+1,h)$.
 The only non-zero
differential is $d_{\ell-1}$, and this is given by $d_{\ell-1}(\xi)=\pm
hx^{h-1}\tau^\ell$, $d_{\ell-1}(t)=\pm \ell x^h \tau^{\ell-1}$. There
are no ungrading problems in the spectral sequence.
\end{theorem}
\begin{proof}
The element $t$ on the $E^2$ page corresponds to the cochain $\tilde t\colon [\,] \mapsto t$
in the Hochschild complex. Applying the formula for the differential,
we have 
\begin{align*}  
(d\tilde t)[\,\underbrace{t,\dots,t}_{\ell-1}\,]&=
m_\ell(\tilde t[\,],t,\dots,t)+m_\ell(t,\tilde
t[\,],\dots,t)+\dots+m_\ell(t,t,\dots,\tilde t[\,])\\
&=\ell m_\ell(t,\dots,t)=\ell x^h. 
\end{align*}
Using this, and the rather simple form of the $A_\infty$ structure on
$B$, it is not hard to see that $d\tilde t = \pm \ell x^h\tau^{\ell-1}$
since the two cochains take the same value on all elements of the bar resolution.

The element $\xi$ on the $E^2$ page corresponds to the cochain
$\tilde\xi\colon [x^i] \mapsto ix^{i-1}$, $[tx^i]\mapsto itx^{i-1}$. Then
applying the formula for the differential, we have
\[ (d\tilde\xi)[\,\underbrace{t,\dots,t}_{\ell}\,]=
\tilde\xi(m_\ell(t,\dots,t))=
\tilde\xi(x^h)=hx^{h-1}. \]
Using this, again we find that $d\tilde\xi = \pm hx^{h-1}\tau^\ell$
since both cochains take the same value on all elements of the bar resolution.

Examining the locations of these terms in the filtration of the bar
complex giving rise to the spectral sequence, we deduce that these
correspond to the differential $d_{\ell-1}$ taking $t$ to $\pm \ell
x^h\tau^{\ell-1}$ and $\xi$ to $\pm hx^{h-1}\tau^\ell$.

Next, we show that the possible values of $n$ for which
$d_n$ is non-zero are very restricted.
The possible tridegrees $(u,v,w)$ at
the $E^2$-term lie in three parallel planes.
Take $N=(\ell -2, \ell, -2a)$ as normal direction and consider the dot
products $N\cdot (u,v,w)=(\ell -2)u+\ell v-2a w$. We have $N\cdot
|x|=0$, $N\cdot |t|=2-\ell$,   $N\cdot |\xi|=2-\ell$, and $N\cdot
|\tau|=0$.
So the only possible values of $N\cdot (u,v,w)$ on the $E^2$ page are
$0$, $2-\ell$, and $4-2\ell$. Furthermore, the elements with $N\cdot
(u,v,w)=4-2\ell$ are multiples of $t\xi$.

The differential $d_n$ decreases $u$ by $n$, increases $v$
by $n-1$, and leaves $w$ unchanged. It therefore increases
$N\cdot (u,v,w)$ by $2n-\ell$. Since $n \ge 2$, we 
first deduce that all differentials are zero on elements with $N\cdot
(u,v,w)=0$, and hence on the polynomial generators $x$ and $\tau$. Next, we
deduce that
the smallest value of $n$ for which $d_n\ne 0$ is 
when $(2n-\ell)+(2-\ell)=0$, so $n=\ell-1$. 
We computed above the value of $d_{\ell-1}$ on the exterior
generators $t$ and $\xi$. Finally, since $h$ and $\ell$ are coprime,
$d_{\ell-1}$ is injective on elements with $N\cdot(u,v,w)=4-2\ell$,
and so there is no room for further differentials.

For the ungrading problem, we note that moving down one place in the
filtration replaces $(u,v,w)$ by $(u-1,v+1,w)$ and so the
dot product with $N$ increases by $N\cdot (-1,1,0)=2$, while $w$ is
unchanged. The relations $hx^{h-1}\tau^\ell=0$ and $\ell
x^h\tau^{\ell-1}=0$ therefore have no ungrading problems, and hold in
$\HH^*B$. The relations $\xi^2=0$ and $t^2=0$ have no ungrading
problems, because there are no candidates with the correct value of
$w$ and with larger dot product with $N$.
\end{proof}

\begin{theorem}\label{th:HHB3cases}
There are three cases for $\HH^*Q\cong\HH^*A\cong\HH^*B$, according
to the characteristic of the field $k$. 
\begin{enumerate}
\item[\rm (i)] If $p\mid h$ then $\HH^*B \cong
k[x,\tau]/(x^h\tau^{\ell-1})\otimes \Lambda(\xi)$.
\item[\rm (ii)]
 If $p\mid \ell$ then $\HH^*B \cong 
k[x,\tau]/(x^{h-1}\tau^\ell)\otimes \Lambda(t)$.
\item[\rm (iii)] If $p\nmid h$ and $p\nmid\ell$ then 
$\HH^*B\cong
(k[x,\tau]\otimes\Lambda(u))/(x^{h-1}\tau^\ell,x^h\tau^{\ell-1},x^{h-1}\tau^{\ell-1} u)$.
\end{enumerate}
Here, we have $|x|=(-2a,-\ell)$, $|t|=(-2b-1,-h)$,
$|\xi|=(2a-1,\ell)$, $|\tau|=(2b,h)$, and $|u|=(-1,0)$.
\end{theorem}
\begin{proof}
If $h>2$ and $\ell>2$ then this follows from Theorem~\ref{th:HHB},
after checking that there are no ungrading problems.
The element $u$ in case (iii) represents ${}\pm ax\xi \pm bt\tau$ in
$E^\infty$ (recall that $ah-b\ell=1$).
If $h=2$ then $B$ is the formal $A_\infty$ algebra $k[x,t]/(t^2+x^h)$,
and we can use the method of Buchweitz and
Roberts~\cite{Buchweitz/Roberts:2015a} to compute $\HH^*B$. If
$\ell=2$ then we can use the same method on $\HH^*A$.
\end{proof}

\begin{corollary}\label{co:x^htau^ell=0}
We have $x^h\tau^\ell=0$ in $\HH^*Q\cong\HH^*A\cong\HH^*B$.
\end{corollary}
\begin{proof}
This follows from Theorem~\ref{th:HHB3cases}:
in all three cases we have $x^h\tau^\ell=0$. Note that if $h>2$ and
$\ell>2$ then we see directly that the differential $d_{\ell-1}$ in Theorem~\ref{th:HHB}
takes $\pm a\xi \pm bt$ (with appropriate signs) to $x^h\tau^\ell$.
\end{proof}

\section{The derived category}\label{sec:Db}

Suppose first that $\fa$ is a DG algebra. 
We write $\D(\fa)$ for the derived
category of $\fa$. This is the triangulated category having 
as objects the left DG $\fa$-modules, and as
arrows the homotopy classes of morphisms of DG modules, with the
quasi-isomorphisms inverted. The shift functor is the suspension
$\Sigma$ defined by $(\Sigma M)_n=M_{n-1}$, so the
triangles take the form $X \to Y\to Z \to \Sigma X$. 

In the case where $H_*\fa$ is a Noetherian graded ring, we write
$\Db(\fa)$ for the thick subcategory of $\D(a)$ whose objects are
the $\fa$-modules $X$ such that $H_*X$ is finitely generated as an
$H_*\fa$-module. We regard this as the analogue of the bounded
derived category in this context; an extended discussion motivating
this can be found in Greenlees and Stevenson~\cite{Greenlees/Stevenson:2020a}.

\begin{definition}
An $\fa$-module is \emph{homotopically projective} if the functor
$\Hom_\fa(X,-)$ preserves quasi-isomorphisms (see
Section~9.1 of Avramov, Foxby and Halperin~\cite{Avramov/Foxby/Halperin:dgha}). Semifree modules
are homotopically projective (Lemma~9.3.5
of~\cite{Avramov/Foxby/Halperin:dgha}) and every module has a semifree
resolution (Theorem~8.3.2 of~\cite{Avramov/Foxby/Halperin:dgha}). It
follows that given any module $X$ there exists a homotopically
projective module $X'$ and a surjective quasi-isomorphism $X' \to X$. We call
this a \emph{homotopically projective resolution} of $X$.
\end{definition}

Homomorphisms in $\D(\fa)$ may be described as follows. Given DG
$\fa$-modules $X$ and
$Y$, choose a homotopically projective module $X'$ and a
quasi-isomorphism $X'\to X$. Then
\[ \Hom_{\D(\fa)}(X,Y) \cong H_*(\Hom_\fa(X',Y)). \]
It can be seen that $X'\to X$ is a \emph{fibrant replacement} with
respect to the projective model structure (see Hovey~\cite{Hovey:1999a})
on $\fa$-modules, and $\D(\fa)$ is the
corresponding homotopy category.\bigskip

Next, we describe the derived category of an $A_\infty$ algebra.
Suppose that $\fa$ is an $A_\infty$ algebra. In this case, the modules
do not form an abelian category, because of the definition of morphism
of $A_\infty$ modules. This time, the derived category $\D(\fa)$ is
the triangulated category having as objects the left $A_\infty$
modules over $\fa$, and as arrows the homotopy classes of $A_\infty$
morphisms. Unlike in the DG context, $A_\infty$ quasi-isomorphisms automatically
have $A_\infty$ inverses. This is again a
triangulated category, with triangles of the form $X \to Y \to Z \to
\Sigma X$. For details, see Keller~\cite{Keller:2001a,Keller:2002a}.
As before, in the case where $H_*\fa$ is Noetherian, 
we write $\Db(\fa)$ for the thick subcategory
whose objects are the modules with finitely generated homology.

In the case where $\fa$ is a DG algebra regarded as an $A_\infty$
algebra with $m_i=0$ for $i>2$, the two definitions agree up to
canonical equivalences of triangulated categories. If $X$ and $Y$ are
DG $\fa$-modules, the homotopy classes of morphisms 
of $A_\infty$ modules from $X$ to $Y$
are canonically isomorphic to the homotopy classes of 
morphisms of DG modules from $X'$ to
$Y$, where $X'$ is a homotopically projective resolution of $X$.
A suitable set of details can be found in Th\'eor\`eme~2.2.2.2 and
Sections~2.4 and~4.1 of the thesis of
Lef\`evre-Hasegawa~\cite{Lefevre-Hasegawa:2003a}.
See also Theorem~4.5 of Keller~\cite{Keller:2006b}.

In the case of the DG algebra $Q$ of Section~\ref{se:Q} and the
$A_\infty$ algebra $A$ of Section~\ref{se:A}, we have the following.

\begin{proposition}\label{pr:DbQsimDbA}
The bounded derived category $\Db(Q)$ is equivalent to $\Db(A)$.
\end{proposition}
\begin{proof}
A quasi-isomorphim of $A_\infty$ algebras induces an equivalence of derived
categories, see for example~\cite{Lefevre-Hasegawa:2003a},
Section~4.1.3. So it follows from
Theorem~\ref{th:QsimA} that $\Db(Q)$ is equivalent to $\Db(A)$.
\end{proof}

\begin{remark}
Although the $A_\infty$ algebras $Q$, $A$ and $B$ carry an internal
grading to make them bigraded, we do not require that the
objects in the derived category carry an internal grading respected by
the morphisms. Nonetheless, we shall make use of internal gradings
in identifying Auslander--Reiten triangles in $\Dsg(B)$ in 
Section~\ref{sec:ARtriangles}. As we shall see, the reason this works
is that the duality established in Section~\ref{se:duality} respects
grading for objects that admit one.
\end{remark}

\section{A spectral sequence}\label{se:Ass}

In this section, we give a brief reminder of the construction and
convergence properties of the spectral
sequence for computing Homs in the derived category $\D(\fa)$ of an $A_\infty$
algebra $\fa$:
\begin{equation}\label{eq:Ass} 
\Ext^{**}_{H_*\fa}(H_*X,H_*Y) \Rightarrow \Hom_{\D(\fa)}(X,Y). 
\end{equation}
We shall make use of this in Section~\ref{se:W} to compute some
endomorphism rings, as a preliminary to applying Auslander--Reiten theory.
The construction is taken from Adams~\cite{Adams:1969a}, and a
discussion of convergence may be found in 
Boardman~\cite{Boardman:1999a}.

Let $\fa$ be an $A_\infty$ algebra. If $X$ is an $A$-module, then
taking homology gives isomorphisms
\[ \Hom_{\D(\fa)}(A,X)\cong\Hom_{H_*\fa}(H_*\fa,H_*X)\cong H_*X. \]
Choosing a set of generators of $H_*X$, we obtain a morphism $F_0\to
X$, where $F_0$ is a direct sum of shifts of $\fa$, with the property
that $H_*F_0\to H_*X$ is surjective. Setting $X_0=X$, we complete to a
triangle 
\[ F_0 \xrightarrow{k} X_0 \xrightarrow{i} X_1\xrightarrow{j} \Sigma F_0 \]
in $\D(\fa)$, and the map $i\colon X_0 \to X_1$ is zero in homology. Repeating
this construction, we obtain a sequence of triangles
\[ \xymatrix{X \ar@{=}[r] & X_0 \ar[rr]^i&&
X_1\ar[rr]^i\ar[dl]^j&&X_2 \ar[rr]^i\ar[dl]^j&& \cdots \\
&&F_0\ar[ul]^k&&F_1\ar[ul]^k&&F_2\ar[ul]^k} \]
where the maps marked $j$ involve a degree shift.  This has the
property that the resulting sequence
\[ \cdots \xrightarrow{(jk)_*} \Sigma^{-2}H_*F_2 \xrightarrow{(jk)_*} \Sigma^{-1}H_*F_1 
\xrightarrow{(jk)_*} H_*F_0 \xrightarrow{k_*}
  H_*X \to 0 \]
is a free resolution of $H_*X$ as an $H_*\fa$-module. 

\begin{lemma}\label{le:colim}
We have $\displaystyle\lim_{\substack{\to \\ i}} X_i \simeq 0$.
\end{lemma}
\begin{proof}
Any map $\Sigma^j\fa \to \displaystyle\lim_{\substack{\to \\ i}} X_i$ factors
through some $X_i$, and then the composite 
\[ \Sigma^j\fa\to X_i\to X_{i+1}\] 
is zero. Thus $H_*\displaystyle\lim_{\substack{\to \\ i}} X_i
=0$ and so $\displaystyle\lim_{\substack{\to \\ i}} X_i \simeq 0$.
\end{proof}

If $Y$ is another $\fa$-module, then taking Homs in $\D(\fa)$ from the
above resolution of $X$ to $Y$, we obtain a diagram of long exact
sequences 
\[ \xymatrix@C=-8mm{
\cdots\ \ \ar[rr]^(.4){i^*} && \Hom_{\D(\fa)}(X_2,Y) \ar[rr]^{i^*}\ar[dl]^(.4){k^*}&& 
\Hom_{\D(\fa)}(X_1,Y)\ar[rr]^{i^*}\ar[dl]^(.4){k^*}&&\Hom_{\D(\fa)}(X_0,Y).\ar[dl]^(.4){k^*} \\
&\Hom_{\D(\fa)}(F_2,Y)&&\Hom_{\D(\fa)}(F_1,Y)\ar[ul]^(.56){j^*} &&\Hom_{\D(\fa)}(F_0,Y)\ar[ul]^(.56){j^*}} \]
The direct sum of all these long exact sequences is an exact couple
\[ \xymatrix@C=-8mm{\Hom_{\D(\fa)}(X_*,Y) \ar[rr] &&
    \Hom_{\D(\fa)}(X_*,Y).\ar[dl] \\
&\Hom_{\D(\fa)}(F_*,Y)\ar[ul]} \]
The spectral sequence of this exact couple has as its $E^1$ term
\[ \Hom_{\D(\fa)}(F_*,Y)\cong\Hom_{\fa}(H_*F_*,H_*Y). \] 
The differential is the composite
\[ \Hom_{\D(\fa)}(F_j,Y) \xrightarrow{(jk)^*} \Hom_{\D(\fa)}(F_{j+1},Y), \]
and so we have
\[ E^2 \cong \Ext^{**}_{H_*\fa}(H_*X,H_*Y), \]
and the abutment of the spectral sequence is
$\Hom_{\D(\fa)}(X,Y)$. Thus we have a filtration
\[ F_j\,\Hom_{\D(\fa)}(X,Y) = \textrm{Image of }\Hom_{\D(\fa)}(X_j,Y)
  \to \Hom_{\D(\fa)}(X,Y), \]
and 
\[ F_j\,\Hom_{\D(\fa)}(X,Y)_n/F_{j+1}\,\Hom_{\D(\fa)}(X,Y)_n \cong
E^\infty_{j,n-j}. \]
In this filtration, $F_0$ is the whole thing, and by Lemma~\ref{le:colim}
we have
\[ \bigcap_j F_j\,\Hom_{\D(\fa)}(X,Y) =
  \lim_{\substack{\leftarrow\\j}}\Hom(X_j,Y) =
\Hom_{\D(\fa)}(\lim_{\substack{\to\\j}}X_j,Y) = 0. \]

As in the spectral sequence for Hochschild cohomology described in
Section~\ref{se:HH}, the spectral sequence is conditionally
convergent, and strongly convergent if and only if 
$\displaystyle\lim_{\substack{\to \\ i}} F_i=0$, which is equivalent
to $\displaystyle\lim_{\substack{\to \\ i}} E^i=0$. In particular, if
each graded piece of $E^i$ is finite dimensional, the the spectral
sequence is strongly convergent.

\section{\texorpdfstring{Inverting $\tau$}{Inverting 𝜏}}\label{se:inverting-tau}

The advantage of the explicit model $Q$ is that the element $\tau$ is
represented by a central cycle, so it is elementary to invert it. 

\begin{definition}
We write $K$ for the graded field $k[\tau, \tau^{-1}]$,  and we define
\[ Q[\tau^{-1}]=K\otimes_{k[\tau]}Q, \] 
as a  DG algebra over $K$.
If $X$ is a DG $Q$-module, we write 
\[ X[\tau^{-1}] = K\otimes_{k[\tau]} X \]
as a DG  $Q[\tau^{-1}]$-module.
\end{definition}

\begin{lemma}\label{le:Hom-tau^-1}
If $X$ and $Y$ are objects in $\Db(Q)$ then 
\[ \Hom_{\Db(Q[\tau^{-1}])}(X[\tau^{-1}],Y[\tau^{-1}])= K\otimes_{k[\tau]}
  \Hom_{\Db(Q)}(X,Y), \]
which we write as $\Hom_{\Db(Q)}(X,Y)[\tau^{-1}]$.
\end{lemma}
\begin{proof}
This follows from the fact that $\tau$ is central in $Q$ with
$d\tau=0$, together with the fact that $H_*(X)$ is finitely generated
over $H_*(Q)$.
\end{proof}

Similarly,
we write $A[\tau^{-1}]$ for $K\otimes_{k[\tau]}A$ as an
$A_\infty$-algebra, and we have a quasi-isomorphism
$Q[\tau^{-1}]\simeq A[\tau^{-1}]$ coming from Theorem~\ref{th:QsimA}.
If $X$ is an $A$-module, we write
$X[\tau^{-1}]=K\otimes_{k[\tau^{-1}]}X$ as an $A[\tau^{-1}]$-module.

\begin{proposition}\label{pr:DbQsimDbAtau^-1}
We have an equivalence of bounded derived categories
\begin{equation*}
\Db(Q[\tau^{-1}])\simeq \Db(A[\tau^{-1}]).
\end{equation*}
If $X$ and $Y$ are objects in $\Db(A)$ then
\[ \Hom_{\Db(A[\tau^{-1}])}(X[\tau^{-1}],Y[\tau^{-1}]) =
  K\otimes_{k[\tau]}\Hom_{\Db(A)}(X,Y) =
  \Hom_{\Db(A)}(X,Y)[\tau^{-1}]. \]
\end{proposition}
\begin{proof}
This follows from Proposition~\ref{pr:DbQsimDbA} and Lemma~\ref{le:Hom-tau^-1}.
\end{proof}

\section{Koszul duality and singularity categories}\label{se:B}

We recapitulate the development of \cite{Greenlees/Stevenson:2020a} in 
our more concrete setting.

\begin{definition}
Let $\fa$ be an $A_\infty$ algebra with Noetherian homology.
\begin{enumerate}
\item[\rm (i)]
The \emph{singularity category}  $\Dsg(\fa)$ is the quotient of $\Db(\fa)$ by
$\Thick(\fa)$. 
\item[\rm (ii)]
The \emph{cosingularity category}  $\Dcsg(\fa)$ is the quotient of $\Db(\fa)$
by $\Thick(k)$.
\end{enumerate}
\end{definition}

\begin{lemma}\label{le:Thick(k)}
Suppose that $\fa$ is an $A_\infty$ algebra such that $H_*\fa$ is
local with residue field $k$.
If $M$ is an $\fa$-module such that $H_*M$ has finite length, then
$M$ is in $\Thick(k)$. 
\end{lemma}
\begin{proof}
A copy of $k$ of lowest degree in $H_*M$ lifts to a map $k\to M$. 
Completing to a triangle $k\to M \to N$, the length of $H_*N$ is
one less than the length of $H_*M$. So by induction on length of
$H_*M$, it follows that $M$ is in $\Thick(k)$.
\end{proof}

For the $A_\infty$ algebras $Q$ and $A$ introduced in
Sections~\ref{se:Q} and~\ref{se:A} and the $A_\infty$ algebras
$Q[\tau^{-1}]$ and $A[\tau^{-1}]$ described in
Section~\ref{se:inverting-tau}  we have the following 
calculation.

\begin{theorem}\label{th:DcsgA}
We have equivalences of triangulated categories
\[ \Dcsg(Q) \simeq \Db(Q[\tau^{-1}]) \simeq \Db(A[\tau^{-1}])\simeq
  \Dcsg(A). \]
\end{theorem}
\begin{proof}
By Propositions~\ref{pr:DbQsimDbA} and~\ref{pr:DbQsimDbAtau^-1}, 
we have compatible equivalences $\Db(Q)\simeq\Db(A)$ and
$\Db(Q[\tau^{-1}])\simeq\Db(A[\tau^{-1}])$.
Since $K$ is flat over $k[\tau]$, we have a functor 
\[ K\otimes_{k[\tau]} - \colon \Db(Q) \to\Db(Q[\tau^{-1}]). \] 
Using Lemma~\ref{le:Thick(k)}, it kills exactly $\Thick(k)$, and therefore induces an
equivalence $\Dcsg(Q)\to \Db(Q[\tau^{-1}])$.
Similarly, 
\[ K \otimes_{k[\tau]} - \colon \Db(A) \to \Db(A[\tau^{-1}]) \]
kills exactly $\Thick(k)$, and therefore induces an equivalence
$\Dcsg(A) \to \Db(A[\tau^{-1}])$.
\end{proof}

\begin{theorem}\label{th:DbAsimDbB}
The functor
$\iHom_A(k,-)$ induces a triangulated equivalence of derived categories 
\[ \Db(A) \xrightarrow{\sim} \Db(B), \] 
that sends $A$ to $k$ and $k$ to $B$. 
It induces triangulated equivalences 
\[ \Dsg(A) \xrightarrow{\sim} \Dcsg(B),\qquad \Dcsg(A)
\xrightarrow{\sim} \Dsg(B). \]
\end{theorem}
\begin{proof}
This follows from Greenlees and Stevenson~\cite{Greenlees/Stevenson:2020a},
Theorem~9.1 and Example~(10.6).
\end{proof}

\begin{corollary}\label{co:DsgBsimDbAtau-1}
We have an equivalence of categories
$\Dsg(B)\simeq\Db(A[\tau^{-1}])$, taking $k$ to $\iEnd_{\Dsg(B)}(k)\cong A[\tau^{-1}]$.
\end{corollary}
\begin{proof}
This follows from Theorems~\ref{th:DcsgA} and~\ref{th:DbAsimDbB}.
\end{proof}

Thus we can regard the central element $\tau$ of $A$ as a periodicity operator on $\Dsg(B)$ of
degree $2b$, namely a natural isomorphism from the identity to
$\Sigma^{2b}$. In Section~\ref{se:indecomposables}, we shall see explicitly how to
interpret this element in terms of resolutions.

\begin{corollary}\label{co:HomDsgB}
If $X$ and $Y$ are objects in $\Db(B)$ then
\[ \Hom_{\Dsg(B)}(X,Y) \cong \Hom_{\Db(B)}(X,Y)[\tau^{-1}]. \]
\end{corollary}
\begin{proof}
This follows from Proposition~\ref{pr:DbQsimDbAtau^-1} using the
equivalence of categories given in Corollary~\ref{co:DsgBsimDbAtau-1}.
\end{proof}

The $A_\infty$ algebras $A[\tau^{-1}]$ and $Q[\tau^{-1}]$ are regular, in the following sense.

\begin{theorem}\label{th:ThickAtau-1}
Every object in $\Db(A[\tau^{-1}])$ is in
$\Thick(A[\tau^{-1}])$.
\end{theorem}
\begin{proof}
Consider the element $x\in B$. For any $B$-module
$M$ with $H_*M$ finitely generated, 
the fibre $F$ of $x\colon M \to \Sigma^{-2a}M$ has the property that
$H_*F$ has finite length. So by Lemma~\ref{le:Thick(k)}, $F$ is in
$\Thick(k)$. Now under the equivalence $\Db(B) \simeq \Db(A)$, the
image of $k$ is $A$. So
regarding $x$ as an element of
$\HH^*B\cong\HH^*A$, we have an action of $x$ on $\Db(A)$, and for any
$A$-module $N$ in $\Db(A)$, the fibre of $x \colon N \to \Sigma^{-2a}N$ is in
$\Thick(A)$.
So we have a triangle
$F \to N \xrightarrow{x} \Sigma^{-2a}N$ with $F$ in $\Thick(A)$.
Inverting $\tau$, we have such a triangle in $\Db(A[\tau^{-1}])$
with $F$ in $\Thick(A[\tau^{-1}])$.
Now by Corollary~\ref{co:x^htau^ell=0},
we have $x^h\tau^\ell=0$ in $\HH^*A$. Since $\tau$ is an isomorphism in 
$\Db(A[\tau^{-1}])$, it follows that $x^h$ acts as zero on $\Db(A[\tau^{-1}])$.
Therefore the fibre of $x^h\colon N \to \Sigma^{-2ah}N$ is 
$N \oplus \Sigma^{-2ah-1}N$. It is also
in $\Thick(F)$, and therefore in $\Thick(A[\tau^{-1}])$. Hence so is $N$.
\end{proof}

\begin{corollary}\label{co:ThickQtau-1}
Every object in $\Db(Q[\tau^{-1}])$ is in $\Thick(Q[\tau^{-1}])$.
\end{corollary}
\begin{proof}
This follows from Theorem~\ref{th:ThickAtau-1}, using the fact that
the equivalence $\Db(A[\tau^{-1}])\simeq\Db(Q[\tau^{-1}])$ of
Theorem~\ref{th:DcsgA} sends $A[\tau^{-1}]$ to $Q[\tau^{-1}]$.
\end{proof}

\begin{corollary}\label{co:dualisable}
If $X$ and $Y$ are objects in $\Db(Q[\tau^{-1}])$ then we have natural equivalences
\begin{enumerate}
\item[\rm (i)] $\Hom_{\Db(Q[\tau^{-1}])}(X,Y) \simeq
\Hom_{\Db(Q[\tau^{-1}])}(X,Q[\tau^{-1}]) \otimes_{Q[\tau^{-1}]} Y$,
\item[\rm (ii)]
$\Hom_{\Db(Q[\tau^{-1}])}(\Hom_{\Db(Q[\tau^{-1}])}(X,Q[\tau^{-1}]),Q[\tau^{-1}])
\simeq X$.
\end{enumerate}
\end{corollary}
\begin{proof}
These hold for $X=Q[\tau^{-1}]$, and hence for every object in
$\Thick(Q[\tau^{-1}])$, which by Corollary~\ref{co:ThickQtau-1} is
every object in $\Db(Q[\tau^{-1}])$.
\end{proof}

\section{\texorpdfstring{Duality for $Q[\tau^{-1}]$-modules}
{Duality for Q[𝜏⁻¹]-modules}}\label{se:duality}

Let $Q[\tau^{-1}]$ be the algebra described in Section~\ref{se:inverting-tau}.
In this section, we prove a form of Tate duality for the
bounded derived category of $Q[\tau^{-1}]$-modules, see
Theorem~\ref{th:Qduality}.  This  combine the dualities
$\Hom_K(-,K)$ (Brown--Comenetz duality) and
$\Hom_{Q[\tau^{-1}]}(-,Q[\tau^{-1}])$ (Spanier--Whitehead duality).
The proof makes essential use of Corollary~\ref{co:dualisable}.

\begin{definition}
We write $Q[\tau^{-1}]^*$ for the dual of $Q$ with respect
to $K$, $\Hom_{k[\tau]}(Q,K)$.
Left and right multiplication make $Q[\tau^{-1}]^*$ into a
$Q$-bimodule, and simultaneously a $K$-module with
compatible actions of $\tau$, and hence a
$Q[\tau^{-1}]$-bimodule. 
Note that 
\begin{align*} 
Q[\tau^{-1}]^* &=
\Hom_{k[\tau]}(Q,K) \cong \Hom_{k[\tau]}(Q,\Hom_K(K,K)) \\
&\cong\Hom_K(K\otimes_{k[\tau]}Q,K) = \Hom_K(Q[\tau^{-1}],K), 
\end{align*}
so we can equally well regard $Q[\tau^{-1}]^*$
as $\Hom_K(Q[\tau^{-1}],K)$.

Similarly, if $X$ is any $K$-vector space (i.e., graded $K$-module),  
we write $X^*$ for $\Hom_K(X,K)$. In particular, 
if $X$ is a left $Q[\tau^{-1}]$-module then $X^*$ is
a right $Q[\tau^{-1}]$-module and if $X$ is a
right $Q[\tau^{-1}]$-module then $X^*$ is a left $Q[\tau^{-1}]$-module.
\end{definition}

\begin{proposition}\label{pr:Qtau-1*}
There is a quasi-isomorphism of 
$Q[\tau^{-1}]$-bimodules $Q[\tau^{-1}] \to \Sigma^{|\xi_1|}
Q[\tau^{-1}]^*$
\end{proposition}

\begin{remark}
For brevity, we have recorded the shift as $|\xi_1|=2a-1$, 
however in the presence of the $\tau$-periodicity, this is only well
defined modulo $2b$. It is more helpful to say that the shift in Tate
duality is one less than the Gorenstein shift ($2a-2b$ in this case) of $A$ as in
\cite[Proposition~4.1]{Greenlees/Stojanoska:2018a}. 

We should also keep track of the internal degrees, and write 
\[ Q[\tau^{-1}] \xrightarrow{\sim} \Sigma^{2a-1,\ell}Q[\tau^{-1}]^*. \]
\end{remark}

\begin{proof}
The standard monomials of Definition~\ref{def:standard} form a
$K$-basis for $Q[\tau^{-1}]$.
We construct a $K$-module homomorphism $Q[\tau^{-1}]\to
\Sigma^{|\xi_1|}Q[\tau^{-1}]^*$ as follows.
It takes all standard monomials to zero except $1$ and $\xi_1$.
It takes $1$ to the element of $Q[\tau^{-1}]^*$ taking value $1$ on $\xi_1$ and
value zero on all other standard monomials, and it takes $\xi_1$ to the element of
$Q[\tau^{-1}]^*$ taking value $1$ on $1$ and value zero on all other standard monomials. 
It is easy to check that
this is a map of $Q[\tau^{-1}]$-bimodules, and using
Corollary~\ref{co:HQ} that it a quasi-isomorphism.
\end{proof}

\begin{proposition}\label{pr:hom-tensor}
If $X$ is a left $Q[\tau^{-1}]$-module and $Y$ is a right
$Q[\tau^{-1}]$-module,  then there is a natural
isomorphism of $K$-vector spaces
\[ \Hom_{Q[\tau^{-1}]}(X,\Hom_K(Y,K)) \cong
  \Hom_K(Y\otimes_{Q[\tau^{-1}]} X,K). \]
If $Y$ is a $Q[\tau^{-1}]$-bimodule, this is an isomorphism of left
$Q[\tau^{-1}]$-modules. 
\end{proposition}
\begin{proof}
This is standard.
\end{proof}

\begin{corollary}\label{co:duals}
If $X$ is homotopically projective then we have a quasi-isomorphism
\[ \Hom_{Q[\tau^{-1}]}(X,Q[\tau^{-1}]) \simeq \Sigma^{|\xi_1|}\Hom_K(X,K). \]
\end{corollary}
\begin{proof}
Using
Propositions~\ref{pr:Qtau-1*} and~\ref{pr:hom-tensor}, and the fact
that $X$ is homotopically projective, we have
\begin{align*}
\Hom_{Q[\tau^{-1}]}(X,Q[\tau^{-1}]) 
&\simeq\Hom_{Q[\tau^{-1}]}(X,\Sigma^{|\xi_1|}Q[\tau^{-1}]^*) \\
&\cong \Sigma^{|\xi_1|}\Hom_{Q[\tau^{-1}]}(X,\Hom_K(Q[\tau^{-1}],K)) \\
&\cong \Sigma^{|\xi_1|}\Hom_K(Q[\tau^{-1}]\otimes_{Q[\tau^{-1}]} X,K) \\
&\cong\Sigma^{|\xi_1|}\Hom_K(X,K).
\qedhere
\end{align*}
\end{proof}

\begin{theorem}\label{th:Qduality}
Let $X$ and $Y$ be objects in $\Db(Q[\tau^{-1}])$.
Then there is a functorial duality of graded $K$-modules
\[ \Hom_{\Db(Q[\tau^{-1}])}(X,Y)^* \cong
\Hom_{\Db(Q[\tau^{-1}])}(Y, \Sigma^{-|\xi_1|}X). \]
\end{theorem}
\begin{proof}
We may replace $X$ and $Y$ by homotopically projective resolutions and
work with $Q[\tau^{-1}]$-module homomorphisms.
Combining Corollary~\ref{co:dualisable}\,(ii) with Corollary~\ref{co:duals}, we have
\[ \Hom_{Q[\tau^{-1}]}(X,Q[\tau^{-1}])^* \simeq \Sigma^{-|\xi_1|}X. \]
Hence using Proposition~\ref{pr:hom-tensor} and Corollary~\ref{co:dualisable}\,(i), we have
\begin{align*}
\Hom_{Q[\tau^{-1}]}(X,Y)^*
&=\Hom_K(\Hom_{Q[\tau^{-1}]}(X,Y),K) \\
&\simeq\Hom_K(\Hom_{Q[\tau^{-1}]}(X,Q[\tau^{-1}])\otimes_{Q[\tau^{-1}]} Y,K) \\
&\cong\Hom_{Q[\tau^{-1}]}(Y,\Hom_K(\Hom_{Q[\tau^{-1}]}(X,Q[\tau^{-1}]),K)) \\
&\simeq\Hom_{Q[\tau^{-1}]}(Y,\Sigma^{-|\xi_1|}X).
\qedhere
\end{align*}
\end{proof}

\begin{corollary}\label{co:Aduality}
Let $X$ and $Y$ be objects in $\Db(A[\tau^{-1}])$. Then there is a
functorial duality of graded $K$-modules 
\[ \Hom_{\Db(A[\tau^{-1}])}(X,Y)^* \cong 
  \Hom_{\Db(A[\tau^{-1}])}(Y,\Sigma^{1-2a}X). \]
If $X$ and $Y$ carry an internal grading, the shift is $\Sigma^{1-2a,-\ell}$.
\end{corollary}
\begin{proof}
This follows from Theorem~\ref{th:Qduality}, using the equivalence of
categories described in\linebreak[3] Proposition~\ref{pr:DbQsimDbAtau^-1}.
\end{proof}

\begin{theorem}\label{th:Bduality}
Let $X$ and $Y$ be objects in $\Dsg(B)$. Then there is a functorial
duality of graded $K$-modules
\[ \Hom_{\Dsg(B)}(X,Y)^* \cong \Hom_{\Dsg(B)}(Y,\Sigma^{1-2a}X). \]
If $X$ and $Y$ carry an internal grading, the shift is $\Sigma^{1-2a,-\ell}$.
\end{theorem}
\begin{proof}
This follows from Corollary~\ref{co:Aduality} using the equivalence of
categories given in Corollary~\ref{co:DsgBsimDbAtau-1}.
\end{proof}

\section{A tour of the two worlds}
\label{se:worlds}
Our aim is to understand the singularity category of $B$, and in 
particular to construct indecomposable objects. In this interlude 
we give a topological account of the strategy before returning to 
give an elementary implementation in algebra. The section can be
entirely ignored by those with a strong algebraic compass. 

Bearing in mind that
$B$ itself is trivial in the singularity category, it is natural to
start with $X_1=k$ and then construct objects from that. Since 
$\Hom_{\Dsg (B)}(k,k)=A[\tau^{-1}]$ it is natural to focus on $\xi$,
which supplies a map $\Sigma^{2a-1}k\to k$, with mapping cone
$X_2$ with cells in degree 0 and $2a$. We then attempt to iterate this
construction by forming 
$$X_s=k\cup_\xi e^{2a}_k\cup_\xi \cdots \cup_\xi e^{2(s-1)a}_k. $$
In fact $X_s$ exists if and only if the $(s-1)$-fold Massey product $\langle \xi, 
\xi, \cdots  ,
\xi\rangle$
exists and contains zero. In our case the indeterminacy is always
zero. We may thus construct $X_1, \ldots, X_{h}$, but not $X_{h+1}$
since the $h$-fold Massey product is nonzero. One may also check inductively that
these complexes are unique up to equivalence. Our main result will 
show that up to suspension this does give all indecomposables. 

The counterparts $Y_s =\Hom_B(k,X_s)$ in $A$-modules will have a cell
structure
\[ Y_s =A\cup_\xi e^{2a}_A\cup_\xi \cdots \cup_\xi  e^{2(s-1)a}_A. \]

Under the derived equivalence between finitely generated $B$-modules
and $A$-modules, small
$B$-modules (i.e., modules in $\Thick(B)$) 
correspond to $A$-modules with finite dimensional
homology, and vice versa. Thus a $B$-module $N$ is small if and only 
if $[k,N]^B_*$ is finite dimensional. Similarly, $N$ corresponds to a small 
$A$-module if and only if $H_*(N)$ is finite dimensional. 

Exchanging the roles of $A$ and $B$, we may construct $A$-modules
$$V_s=k\cup_t e^{2b}_k\cup_t \cdots \cup_t e^{2(s-1)b}_k. $$
for $s=1, 2, \ldots , \ell$ but not $V_{\ell +1}$ since the
$\ell$-fold Massey product is nonzero. These correspond
to the small $B$-modules
$$W_s= B \cup_t e^{2b}_B\cup_t \cdots \cup_t e^{2(s-1)b}_B. $$

Our actual method will take full advantage of the explicit and easily described
models, and give concrete representatives for the objects $W_s$.

\begin{remark}
\label{rem:formal}
For the graded rings $H_*(A)$ with $h>2$ and $H_*(B)$ with $\ell>2$ 
(i.e.,  the formal case  where the algebras are polynomial tensor
exterior and all 
products $m_i$ are zero for $i>2$) the singularity and cosingularity
categories are  well understood, for example through the theory of
maximal Cohen--Macaulay modules. See for 
example Proposition~2.6 of Ene and Popescu~\cite{Ene/Popescu:2008a},
which discusses modules over the ungraded completion, but the  modules
over the graded ring are very similar.  In that case 
the objects $X_s, Y_s, V_s, W_s$ exist for all $s\ge 0$ and the terms in each sequence are 
inequivalent.  By contrast with our case, the singularity and 
cosingularity categories for $H_*(A)$ and $H_*(B)$ each contain 
infinitely many indecomposable objects. 
\end{remark}

\section{The $B$-modules $W_i$  in $\Thick(B)$}\label{se:W}

Following the pattern set out in Section \ref{se:worlds} we will construct
explicit small $B$-modules (i.e., modules in $\Thick(B)$) 
$W_1, W_2, \ldots , W_\ell$, starting with
$W_1=B$, and we will do so in a bigraded fashion. Background on modules
over $A_\infty$-algebras may be found in \cite{Keller:2006b}.

Consider the map $t\colon \Sigma^{-2b-1,-h}B \to B=W_1$. Complete this to a
triangle in $\Db(B)$, 
\[ \Sigma^{-2b-1,-h}B \xrightarrow{t} W_1 \to W_2. \]
Then using the long exact sequence in homology, as a module over
$H_*B=k[x]\otimes\Lambda(t)$ we have
\[ H_*W_2= k[x].u_2 \oplus k[x].v_2, \] 
where $k[x]=H_*B/(t)$, $|u_2|=(0,0)$ and $|v_2|=(-4b-1,-2h)$. The $A_\infty$
structure is given by 
\begin{align*} 
m_3(t,t,u_2)&=v_2 \\
m_{\ell-1}(t,\dots,t,v_2)&=(-1)^{(\ell-2)(\ell-1)/2}x^hu_2.
\end{align*}
The element $v_2$ defines a map $\Sigma^{-4b-1}B\to W_2$, which we
complete to a triangle 
\[ \Sigma^{-4b-1,-2h}B \xrightarrow{v_2} W_2 \to W_3. \]
We have
\[ H_*W_3=k[x].u_3 \oplus k[x].v_3, \]
with $|u_3|=(0,0)$, $|v_3|=(-6b-1,-3h)$. The $A_\infty$ structure is given by
\begin{align*}
m_4(t,t,t,u_3)&=v_3 \\
m_{\ell-2}(t,\dots,t,v_3)&=(-1)^{(\ell-4)(\ell-1)/2}x^hu_3.
\end{align*}
Continuing this way, we construct objects $W_i$ in $\Db(B)$ ($1\le
s\le \ell$), finitely
built from $B$, with
\[ H_*W_i = k[x].u_i \oplus k[x].v_i, \]
with $|u_i|=(0,0)$, $|v_i|=(-2ib-1,-ih)$, and
\begin{align*}
m_{i+1}(\underbrace{t,\dots,t}_{\text{$i$ copies}},u_i)&=v_i \\
m_{\ell-i+1}(\underbrace{t,\dots,t}_{\text{$\ell-i$ copies}},v_i)&=(-1)^{(\ell-2i+2)(\ell-1)/2}x^hu_i.
\end{align*}
and $v_i$ defines a triangle
\[ \Sigma^{-2ib-1,-ih}B \stackrel{v_i}\to W_i \to W_{i+1}. \]
The second to last stage of this process is a module $W_{\ell-1}$ with
\begin{align*}
m_\ell(t,\dots,t,u_{\ell-1})&=v_{\ell-1} \\
tv_{\ell-1}=m_2(t,v_{\ell-1})&=(-1)^{\ell(\ell-1)/2}x^hu_{\ell-1}.
\end{align*}
Then something different happens. The map $\Sigma^{-2(\ell-1)b-1,-(\ell-1)h}B\to
W_{\ell-1}$ still defines a triangle
\[ \Sigma^{-2(\ell-1)b-1,-(\ell-1)h}B\to
W_{\ell-1} \to W_\ell, \]
but the long exact sequence in homology now gives 
$H_*W_\ell \cong k[x]/(x^h)$, and the process cannot be continued any
further. The composite $B=W_1\to W_2\to \cdots \to W_\ell$ shows that
$W_\ell$ is equivalent to the quotient of $B$ by the ideal $(x^h,t)$. What we have seen is the
following.

\begin{theorem}\label{th:kx/x^h}
The $B$-module $k[x]/(x^h)=B/(x^h,t)$ is in the subcategory $\Thick(B)$
of $\Db(B)$. It is built from $\ell$ shifts of copies of $B$.\qed
\end{theorem}

\section{Indecomposables objects in $\Dsg(B)$}\label{se:indecomposables}

In this section, we discuss the indecomposable objects which we eventually
wish to show form a complete list in $\Dsg(B)$. They come with an internal
grading,  which we keep track of.
The ideal in $B$ generated by $x^h$ 
and $t$ is a bigraded $A_\infty$ ideal, and the quotient 
\[ \bar B=B/(x^h,t)\cong k[x]/(x^h) \] 
is a formal $A_\infty$ algebra with $|x|=(-2a,-\ell)$. In other words, the only
non-zero operation $m_i$ on the quotient is the multiplication $m_2$.
We begin with a discussion of $\bar B$-modules. The indecomposable
modules are the shifts of the quotients $\bar M_i=k[x]/(x^i)$ for
$1\le i\le h$ of $\bar B$.
For $i<j$ we have short exact sequences of $\bar B$-modules
\begin{equation}\label{eq:Mj-i}
 0 \to \Sigma^{-2ia,-i\ell}\bar M_{j-i} \to\bar M_j \to \bar M_i \to 0. 
\end{equation}
We also have almost split sequences of $\bar B$-modules
\begin{gather*}
0 \to \Sigma^{-2a,-\ell}\bar M_1 \to\bar M_2 \to \bar M_1 \to 0  \\
0 \to \Sigma^{-2a,-\ell}\bar M_2 \to \Sigma^{-2a,-\ell}\bar M_1\oplus \bar M_3 \to
M_2 \to 0 \\
\cdots \\
0 \to \Sigma^{-2a,-\ell}\bar M_{h-1}\to \Sigma^{-2a,-\ell}\bar M_{h-2}\oplus 
\bar M_h \to \bar M_{h-1} \to 0.
\end{gather*}

Pulling back the $k[x]/(x^h)$-module $\bar M_i$ ($1\le i\le h$) to $B$, we obtain a
$B$-module $M_i$ and a map $B \to M_i$. 
In particular, $M_1$ is the field
object, with $H_*M_1\cong k$. We shall write $k$
for the $B$-module $M_1$. By Theorem~\ref{th:kx/x^h}, $M_h$ is in 
the subcategory $\Thick(B)$ of $\Db(B)$, and therefore vanishes in $\Dsg(B)$.
We write $X_i$ for the fibre of $B \to M_i$.
Applying $H_*$
to the triangle $X_i \to B \to M_i$, we see that $H_*X_i$ is the
ideal generated by $t$ and $x^i$ in $H_*B=k[x]\otimes \Lambda(t)$. 
In $\Dsg(B)$, we have $\Sigma X_i \simeq M_i$  and $X_h\simeq 0$.

The exact sequence~\eqref{eq:Mj-i} of $\bar B$-modules gives rise to
triangles in $\Db(B)$
\[ \Sigma^{-2ia,-i\ell}X_{j-i}\to X_j \to X_i. \]
In particular, taking $j=h$, we obtain the following, critical in
linking odd and even suspensions. 

\begin{proposition}\label{pr:X_{h-i}}
There is an equivalence
$ X_i \simeq \Sigma^{1-2ia,-i\ell}X_{h-i}$ in $\Dsg(B)$. \qed 
\end{proposition}

Similarly, the almost split sequences of $\bar B$-modules give
rise to triangles in $\Db(B)$
\begin{gather}
\Sigma^{-2a,-\ell}X_1 \to X_2 \to X_1 \notag \\
\Sigma^{-2a,-\ell}X_2 \to \Sigma^{-2a,-\ell}X_1\oplus X_3 \to X_2 \notag \\
\cdots \label{eq:ass} \\
\Sigma^{-2a,-\ell}X_{h-1}\to \Sigma^{-2a,-\ell}X_{h-2}\oplus X_h \to X_{h-1}.\notag
\end{gather}
Bearing in mind that $X_h$ is zero in $\Dsg(B)$, the last of these
becomes
\[ \Sigma^{-2a,-\ell}X_{h-1} \to \Sigma^{-2a,-\ell}X_{h-2} \to X_{h-1}. \]
We shall eventually see that these are almost split triangles in $\Dsg(B)$.

Next, we compute the spectral sequence \eqref{eq:Ass}
\begin{equation}\label{eq:ss} 
\Ext^{*,*}_{H_*B}(H_*X_i,H_*X_j)\Rightarrow\Hom_{\Db(B)}(X_i,X_j).
\end{equation}
We first assume that $\ell>2$, and later describe the modifications
necessary for the case $\ell=2$.
Resolving $H_*X_i$ as a module over $H_*B$, we obtain the sequence
\begin{multline*} 
\cdots \xrightarrow{\left(\begin{smallmatrix} t & x^i \\ 0 & -t \end{smallmatrix}\right)}
\Sigma^{-4b-2,-2h}H_*B \oplus \Sigma^{-2ia-2b-1,-h-i\ell}H_*B \\
\xrightarrow{\left(\begin{smallmatrix} t & x^i \\ 0 & -t \end{smallmatrix}\right)}
\Sigma^{-2b-1,-h}H_*B \oplus \Sigma^{-2ia,-i\ell}H_*B \xrightarrow{(t,x^i)} H_*X_i
\to 0. 
\end{multline*}
Thus $H_*X_i$ is a maximal Cohen--Macaulay $H_*B$-module corresponding
to the following matrix factorisation of the relation $t^2=0$.
\[ \begin{pmatrix} t & x^i \\ 0 & -t\end{pmatrix} 
\begin{pmatrix} t & x^i \\ 0 & -t\end{pmatrix} 
= t^2\begin{pmatrix} 1 & 0 \\ 0 & 1 \end{pmatrix} \]

If $X_i$ and $X_j$ are $B$-modules of this form, we take homomorphisms
from the resolution of $H_*X_i$ to $H_*X_j$ to obtain
\[ \cdots 
\xleftarrow{\left(\begin{smallmatrix}t&0\\x^i&-t\end{smallmatrix}\right)} 
\Sigma^{4b+2,2h}H_*X_{j} \oplus \Sigma^{2ia+2b+1,h+i\ell}H_*X_{j}
\xleftarrow{\left(\begin{smallmatrix}t&0\\x^i&-t\end{smallmatrix}\right)} 
\Sigma^{2b+1,h}H_*X_{j} \oplus \Sigma^{2ia,i\ell}H_*X_{j} \leftarrow 0.
\]
The differential sends
\[ \left(\begin{smallmatrix}t\\0\end{smallmatrix}\right) \to
\left(\begin{smallmatrix}0\\tx^i\end{smallmatrix}\right) \qquad
\left(\begin{smallmatrix}x^{j}\\0\end{smallmatrix}\right) \to
\left(\begin{smallmatrix}tx^{j}\\x^{i+j}\end{smallmatrix}\right) \qquad
\left(\begin{smallmatrix}0\\t\end{smallmatrix}\right) \to
\left(\begin{smallmatrix}0\\0\end{smallmatrix}\right) \qquad
\left(\begin{smallmatrix}0\\x^{j}\end{smallmatrix}\right) \to
\left(\begin{smallmatrix}0\\-tx^{j}\end{smallmatrix}\right).
\]

\subsubsection*{\bf (i) The case $j \ge i$:}

The kernel is generated by 
$\left(\begin{smallmatrix}0\\t\end{smallmatrix}\right)$
and
\[ x^{j-i}\left(\begin{smallmatrix}t\\0\end{smallmatrix}\right) +
\left(\begin{smallmatrix}0\\x^{j}\end{smallmatrix}\right) =
\left(\begin{smallmatrix}tx^{j-i}\\x^{j}\end{smallmatrix}\right). \]
The image contains $t$ times these and $x^{i}$ times these.

We write $\alpha$ for the element of the kernel represented by 
the first of these elements $\left(\begin{smallmatrix}0\\t\end{smallmatrix}\right)$, and $\beta$ for the element represented 
by the second $\left(\begin{smallmatrix}tx^{j-i}\\x^{j}\end{smallmatrix}\right)$. Thus as a module over $H_*B$,
the degree zero homology of the above complex is $\Hom_{H_*B}(H_*X_i,H_*X_j)$, and is
generated by $\alpha$ and $\beta$ with relation $x^j\alpha=t\beta$.
Recalling that $H_*(A)= \Hom_{\Db(B)}(X_1,X_1)$, the periodicity
element of degree $(-1,2b+1,h)$ in the above resolution represents the element
$\tau\in H_{2b}A$, so we shall call it $\tau$ by abuse of notation. 
As a module over $H_*B[\tau]$ we have
\[ \Ext^{*,*}_{H_*B}(H_*X_i,H_*X_{j})=\langle \alpha,\beta\rangle/
(t\alpha,x^{j}\alpha-t\beta,t\tau\beta,x^i\tau\beta) \]
with $\alpha$ in degree $(0,2ja-2b-1,j\ell-h)$ and $\beta$ in degree $(0,0,0)$.

\subsubsection*{\bf (ii) The case $i\ge j$:}

The kernel is generated by 
$\left(\begin{smallmatrix}0\\t\end{smallmatrix}\right)$ and
\[ \left(\begin{smallmatrix}t\\0\end{smallmatrix}\right) +
x^{i-j}\left(\begin{smallmatrix}0\\x^{j}\end{smallmatrix}\right) =
\left(\begin{smallmatrix}t\\x^{i}\end{smallmatrix}\right). \]
The image contains $t$ times these, and $x^{j}$
times these.

We again write $\alpha$ for the element of the kernel represented by
the first of these $\left(\begin{smallmatrix}0\\t\end{smallmatrix}\right)$ 
and $\beta$ for the element represented by the
second
$\left(\begin{smallmatrix}t\\x^{i}\end{smallmatrix}\right)$.
This time, 
\[ \Ext^{*,*}_{H_*B}(H_*X_i,H_*X_{j})=\langle \alpha,\beta\rangle/
(t\alpha,x^{i}\alpha-t\beta,t\tau\beta,x^j\tau\beta) \]
with $\alpha$ in degree $(0,2ia-2b-1,i\ell-h)$ and $\beta$ in degree $(0,0,0)$.

\subsubsection*{\bf (iii) The case $i=j$}

The cases $i=j$ in (i) and (ii) coincide, but there is one extra
piece of structure to determine, namely the ring structure. This
is determined by the statements that $\beta$ is the identity
endomorphism of $X_i$, and $\alpha$ is the endomorphism
sending $x^i$ to $t$ and $t$ to zero. So the extra relations
are $\alpha^2=0$ and $\beta=1$. Thus $x^i\alpha=t$, so $t$ is a
redundant generator, and the presentation becomes
\begin{equation}\label{eq:XiE2} 
\Ext^{*,*}_{H_*B}(H_*X_i,H_*X_i)\cong
  k[x,\tau,\alpha]/(\alpha^2,x^i\tau). 
\end{equation}

For the computations above, we assumed that $\ell>2$. We now explain
the modifications necessary for the case $\ell=2$. In this case, the
matrix in the resolution  of $H_*X_i$ is
$ \left(\begin{smallmatrix} t & x^i \\ x^{h-i} & -t \end{smallmatrix}\right)$.
This corresponds to the following matrix factorisation of the relation
$t^2+x^h=0$.
\[ \begin{pmatrix} t & x^i \\ x^{h-i} & -t \end{pmatrix}
\begin{pmatrix} t & x^i \\ x^{h-i} & -t \end{pmatrix}= (t^2+x^h)
\begin{pmatrix} 1 & 0 \\ 0 & 1 \end{pmatrix}. \]
Taking homomorphisms into $H_*X_j$, the matrix for the differential is
the transpose,
$\left(\begin{smallmatrix} t & x^{h-i} \\ x^i & -t \end{smallmatrix}\right)$.
The differential sends
\[ \left(\begin{smallmatrix}t\\0\end{smallmatrix}\right) \to
\left(\begin{smallmatrix}-x^h\\tx^i\end{smallmatrix}\right) \qquad
\left(\begin{smallmatrix}x^{j}\\0\end{smallmatrix}\right) \to
\left(\begin{smallmatrix}tx^{j}\\x^{i+j}\end{smallmatrix}\right) \qquad
\left(\begin{smallmatrix}0\\t\end{smallmatrix}\right) \to
\left(\begin{smallmatrix}tx^{h-i}\\x^h\end{smallmatrix}\right) \qquad
\left(\begin{smallmatrix}0\\x^{j}\end{smallmatrix}\right) \to
\left(\begin{smallmatrix}x^{h+j-i}\\-tx^{j}\end{smallmatrix}\right).
\]
There are more cases this time, but
we content ourselves with computing $\Ext^{*,*}_{H_*B}(X_i,X_j)$ in
the case $i=j\le h/2$, which is all we need. 
In this case, the kernel is generated by the element
$\alpha$ representing 
$\left(\begin{smallmatrix} -x^{h-i} \\ t \end{smallmatrix}\right)$ 
and $\beta$ representing
$\left(\begin{smallmatrix} t \\ x^i \end{smallmatrix}\right)$.
Again $\beta$ is the identity endomorphism of $X_i$
and $\alpha$ is the endomorphism sending $x^i$ to $t$ and $t$ to
$-x^{h-i}$, so $\alpha^2$ is multiplication by $-x^{h-2i}$. This time
the presentation becomes
\begin{equation}\label{eq:Xi,l=2} 
\Ext^{*,*}_{H_*B}(H_*X_i,H_*X_i) \cong
  k[x,\tau,\alpha]/(\alpha^2+x^{h-2i},x^i\tau). 
\end{equation}

\begin{theorem}\label{th:EndDbXi}
If $2i \le h$, we have 
\[ \End_{\Db(B)}(X_i)\cong
k[x,\tau,\alpha]/(\alpha^2+\lambda x^{h-2i}\tau^{\ell-2},x^i\tau) \] 
with $|x|=(-2a,-\ell)$,
$|\tau|=(2b,h)$, $|\alpha|=(2ia-2b-1,i\ell-h)$, and $\lambda$ some
scalar in $k$.
\end{theorem}
\begin{proof}
First assume that $\ell > 2$.
We have to show that there are no non-zero differentials
in the spectral sequence
\[ \Ext^{*,*}_{H_*B}(H_*X_i,H_*X_i) \Rightarrow \End_{\Db(B)}(X_i) \]
(see Section~\ref{se:Ass}) whose $E^2$ page is given 
by~\eqref{eq:XiE2}, and we then have to address the ungrading
problem for the $E^\infty$ term.
This is easier if we take into
account the internal degrees, which have to be preserved by the
differentials. So elements in the spectral sequence are triply graded,
with $|x|=(0,-2a,-\ell)$, $|\tau|=(-1,2b+1,h)$, 
$|\alpha|=(0,2ia-2b-1,i\ell-h)$. As in the proof of
Theorem~\ref{th:HHB}, the possible tridegrees $(u,v,w)$ at
the $E^2$-term are very restricted. This time they lie in two parallel planes.
We use the same normal vector $N=(\ell -2, \ell, -2a)$, and consider the dot
products $N\cdot (u,v,w)=(\ell -2)u+\ell v-2a w$. We note that $N\cdot |x|=0, N\cdot |\tau|=0$
and  $N\cdot |\alpha|=2-\ell$. 
Hence  $N\cdot |x^i\tau^j|=0$ and $N\cdot |x^i\tau^j\alpha|=2-\ell$,
and these are   the only possible values of $N\cdot (u,v,w)$ for degrees of
elements at $E^2$.  

The differential $d_n$ decreases $u$ by $n$, increases $v$
by $n-1$, and leaves $w$ unchanged. It therefore increases
$N\cdot (u,v,w)$ by $2n-\ell$. Since $n \ge 2$, we deduce that
either $d_n=0$ or $(2n-\ell)+ (2-\ell) = 0$, so $n=\ell-1$.
In the latter case, $d_{\ell-1}$ sends $x$ and $\tau$ to zero,
and $\alpha$ to an element of degree
$(-\ell+1,2ia-2b+\ell-3,i\ell-h)$. There is only one such monomial in $x$
and $\tau$, namely $x^{h-i}\tau^{\ell-1}$. The assumption that 
$2i \le h$ implies that this is zero, since $x^i\tau=0$.

For the ungrading problem, since $\alpha^2$ has even degree
it cannot involve $\alpha x^{i_1}\tau^{i_2}$. If $\alpha^2$ has the same bidegree as
$x^{i_1}\tau^{i_2}$ then equating bidegree and solving, we find that
$i_1=h-2i$, $i_2=\ell-2$. So we have that $\alpha^2$
is a multiple of $x^{h-2i}\tau^{\ell-2}$.

The element $x^i\tau$ has degree
$(-1,2b-1-2ia,h-i\ell)$. There are no non-zero monomials in $E^\infty$
with this internal degree, so $x^i\tau$ ungrades to zero.

Finally, in the case $\ell=2$, $B$ and $X_i$ are formal, and the $E^2$
page~\eqref{eq:Xi,l=2} of
the spectral sequence computes $\End_{\Db(B)}(X_i)$ with no non-zero
differentials or ungrading problems.
\end{proof}

\begin{remark}\label{rk:lambda}
If $\ell>2$ and $i\le h/3$  then the element $x^{h-2i}\tau^{\ell-2}$ is zero, so
the value of $\lambda$ only matters when $h/3 < i \le h/2$. 
Using the models we produce in Section~\ref{se:models}, it turns out
that $\lambda=1$ always holds, as we saw above in the case $\ell=2$. 
We shall not need this information in what follows.
\end{remark}

\begin{theorem}\label{th:EndDsgXi}
We have 
\[ \End_{\Dsg(B)}(X_i)\cong 
\begin{cases} K[x,\alpha]/(\alpha^2+\lambda x^{h-2i}\tau^{\ell-2},x^i) & 2i \le h \\
K[x,\alpha]/(\alpha^2+\lambda x^{2i-h}\tau^{\ell-2},x^{h-i}) & 2i \ge h
\end{cases} \] 
(recall $K=k[\tau,\tau^{-1}]$). This has finite length over $K$, and 
in particular, it is an Artinian local graded ring.
The (homological) degree zero part is $k[x^{|b|}\tau^{|a|}]$; note
that $x^{|b|}\tau^{|a|}$ is nilpotent, and often zero.
As a module over $\End_{\Dsg(B)}(X_i)$, $\Hom_{\Dsg(B)}(X_i,X_j)$ has
finite length. The Krull--Schmidt theorem holds for finite direct sums
of shifts of the objects $X_i$.
\end{theorem}
\begin{proof}
By Corollary~\ref{co:HomDsgB} we have
$\End_{\Dsg(B)}(X_i)\cong \End_{\Db(B)}(X_i)[\tau^{-1}]$, so if $2i
\le h$ it
follows from Theorem~\ref{th:EndDbXi} that this is isomorphic to
$K[x,\alpha]/(\alpha^2+\lambda x^{2h-i}\tau^{\ell-2},x^i)$. The structure for $2i \ge h$ 
then follows from the isomorphism $X_i \cong \Sigma^{1-2ia}X_{h-i}$
given in Proposition~\ref{pr:X_{h-i}}.

Next we compute the homological degree zero part.
The degree of $\alpha$ is $2ia-2b-1$, which is odd, while
the degrees of $x$ and $\tau$ are even. If $x^j\tau^m$ has degree zero
then $ja=mb$. Since $a$ and $b$ are coprime, this implies that $j$ is
divisible by $b$, $m$ is divisible by $a$, and the element is a power
of $x^{|b|}\tau^{|a|}$ (note that $ah-b\ell=1$, and $h$ and $\ell$ are
positive, $a$ and $b$ have the same sign).

The $E^2$ page of the spectral sequence \eqref{eq:ss} is finitely
generated over $k[\tau]$, and therefore
$\Hom_{\Dsg(B)}(X_i,X_j)=\Hom_{\Db(B)}(X_i,X_j)[\tau^{-1}]$ has finite
length over $K$.

The final statement about the Krull--Schmidt theorem follows from the
fact that the endomorphism rings of the $X_i$ are local rings.
\end{proof}

\begin{corollary}\label{co:socle}
The socle of the homological degree $2a-1$ part of 
$\End_{\Dsg(B)}(X_i)$, as a module over the degree
zero part, is spanned 
by the element $\alpha x^{i-1}\tau$ in degree $(2a-1,\ell)$.
None of the other monomials of homological degree $2a-1$ has internal degree $\ell$.
\end{corollary}
\begin{proof}
Using Theorem~\ref{th:EndDsgXi},
the monomials of homological degree $2a-1$ in $\End_{\Dsg(B)}(X_i)$
are the elements $\alpha x^{i-1-j|b|}\tau^{1-j|a|}$ in degree $(2a-1,\ell+j(\ell|b|-h|a|))=(2a-1,\ell\pm j)$.
\end{proof}

\begin{corollary}\label{co:socle-ass}
The socles of the homological degree $2a-1$ parts of 
$\End_{\Dsg(B)}(X_i)$ for $1\le i\le h$ are represented by
the triangles~\eqref{eq:ass} coming from the almost split sequences of
$\bar B$-modules.
\end{corollary}
\begin{proof}
The connecting homomorphism for the triangle~\eqref{eq:ass} for $X_i$
is a
non-zero element of degree $(2a-1,\ell)$ in $\End_{\Dsg(B)}(X_i)$.
By Corollary~\ref{co:socle}, it is therefore a non-zero multiple of
$\alpha x^{i-1}\tau$.
\end{proof}

\begin{remark}
The images in $\Db(A)$ of the objects $X_i$ of 
$\Db(B)$ are, up to shifts, the analogues of
the modules $W_i$ described in Section~\ref{se:W}.

Write $Y_i=\iHom_B(k,X_i)$ for the image of $X_i$ in $\Db(A)$.
Then $Y_i$ is a free $k[\tau]$-module on two 
generators, $u_i$ of degree $2a-1$ and 
$v_i$ in degree $-2(i-1)a$. The action of $\xi$ is given by 
\begin{align*}
m_{i+1}(\underbrace{\xi,\dots,\xi}_{\text{$i$ copies}},u_i)&=v_i, \\
m_{h-i+1}(\underbrace{\xi,\dots,\xi}_{\text{$h-i$ copies}},v_i)&=(-1)^{(h-2i+2)(h-1)/2}\tau^\ell u_i
\end{align*}
and the remaining $m_j$ vanish.
\end{remark}

\section{Auslander--Reiten triangles}
\label{sec:ARtriangles}

For background on Auslander--Reiten triangles in triangulated
categories, see Happel~\cite{Happel:1987a,Happel:1988a}. We shall
construct them in the category $\Dsg(B)$ and use them to classify the
indecomposables. Although we make use of the internal grading to
identify these triangles, the grading does not interfere with their
existence and use for classification.

Suppose that $X$ is an indecomposable, internally gradable object in
$\Dsg(B)$ with local endomorphism ring
$\End_{\Dsg(B)}(X)$. A map from another, not necessarily internally 
gradable object $Y$ to $X$
has a right inverse (i.e., induces an isomorphism from a direct summand of
$Y$ to $X$) if and only if the induced map
$\Hom_{\Dsg(B)}(X,Y) \to \End_{\Dsg(B)}(X)$ is surjective.

Using Theorem~\ref{th:Bduality}, the dual
$\Hom_{\Dsg(B)}(X,\Sigma^{1-2a}X)$ has a simple socle as a
module over
$\End_{\Dsg(B)}(X)$, and the socle is a map from $X$ to $\Sigma^{1-2a,-\ell}X$.
For the objects $X_i$, this socle is identified
in Corollary~\ref{co:socle}.
Choose a non-zero morphism
$\alpha_X\colon X \to \Sigma^{1-2a,-\ell}X$ in this socle. This has the
property that a map $Y\to X$ has a right inverse if and only if the
composite $Y \to X \xrightarrow{\alpha_X} \Sigma^{1-2a}X$ is non-zero. 

Complete to a triangle
\[ \Sigma^{-2a,-\ell} X \to Z \to X \xrightarrow{\alpha_X} \Sigma^{1-2a,-\ell}X \]
in $\Db(A[\tau^{-1}])$. This then has the following lifting property. If a map
$Y\to X$ does not have a right inverse, then it lifts to a map $Y\to
Z$.
\[ \xymatrix@+3mm{&&Y \ar[d] \ar@{.>}[dl] & \\
\Sigma^{-2a}X \ar[r] & Z \ar[r] & X \ar[r]^(.4){\alpha_X} & \Sigma^{1-2a}X} \]
Similarly, if $\Sigma^{-2a}X \to Y$ does not have a left inverse, then  
it extends to a map $Z \to Y$.
\[ \xymatrix@+3mm{\Sigma^{-1}X\ar[r]^{\Sigma^{-1}\alpha_X} 
&\Sigma^{-2a}X \ar[r] \ar[d] & Z \ar[r] \ar@{.>}[dl] 
& X \ar[r]^(.4){\alpha_X} & \Sigma^{1-2a}X \\ &Y} \]
These are the defining properties of an \emph{Auslander--Reiten triangle},
sometimes also called an \emph{almost split triangle}.

\begin{theorem}
The Auslander--Reiten triangles in $\Dsg(B)$ for the objects $X_i$ are
the triangles~\eqref{eq:ass}, and the Auslander--Reiten translate is
$\Sigma^{-2a,-\ell}$. 
\end{theorem}
\begin{proof}
This follows from Corollary~\ref{co:socle-ass}.
\end{proof}

If $Z'$ is a direct summand of $Z$ then the composite $Z' \to Z \to X$
is called an \emph{irreducible morphism}. It has the property that it
is not an isomorphism, but for any factorisation $Z' \to U \to X$,
either $Z'\to U$ has a left inverse or $U \to X$ has a right inverse.
This property is symmetric, so that if $Z'$ is a direct summand of $Z$
then the composite $\Sigma^{-2a}X \to Z \to Z'$ is also an irreducible
morphism. 

The \emph{Auslander--Reiten quiver} is the quiver (directed graph)
whose vertices correspond to the isomorphism classes of
indecomposable objects and whose directed edges correspond to the 
irreducible morphisms. This comes with a translation $\Sigma^{-2a}$
with the property that there is an arrow from $Z$ to $X$ if and only
if there is an arrow from $\Sigma^{-2a} X$ to $Z$.

For the objects $\Sigma^jX_i$,  the Auslander--Reiten quiver has the following form.
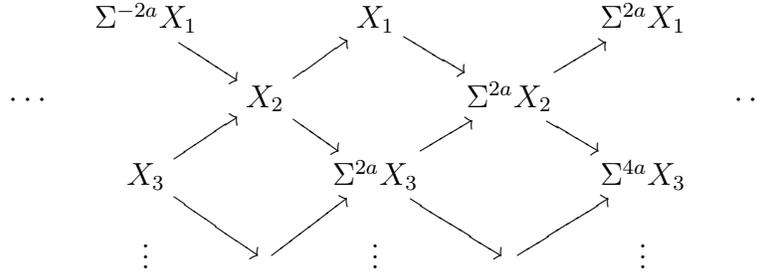
\[ 
\xymatrix@=4mm{
&\Sigma^{-2a}X_1\ar[dr] && X_1 \ar[dr] && \Sigma^{2a}X_1 &\\
\cdots&&X_2\ar[ur]\ar[dr]&&\Sigma^{2a}X_2\ar[ur]\ar[dr]&&\cdots\\
&X_3\ar[ur]\ar[dr]&&\Sigma^{2a}X_3\ar[ur]\ar[dr]&&\Sigma^{4a}X_3&\\
&\vdots&\ar[ur]&\vdots&\ar[ur]&\vdots&} 
\]

This wraps round to form a cylinder, since $\Sigma^{2b}X_i$ is
isomorphic to $X_i$ via $\tau$. Since $a$ and $b$ are coprime, the
circumference of the cylinder is $b$. The height is $h-1$,
and $X_{h-1}\cong \Sigma^{2a-1}X_1$ is in the bottom row.
This gives a total of $b(h-1)$ isomorphism classes of objects.
In the usual language of translation quivers, we therefore have the
following. 

\begin{theorem}\label{th:AR-quiver}
The objects $\Sigma^jX_i$ in $\Dsg(B)$ form a connected component 
of the Auslander--Reiten quiver, isomorphic to $\bZ
A_{h-1}/T^{|b|}$, where $T$ is the translation functor.\qed
\end{theorem}

For example, if $a=6$, $h=6$, $b=5$, $\ell=7$, we get the following
quiver isomorphic to $\bZ A_5/T^5$.
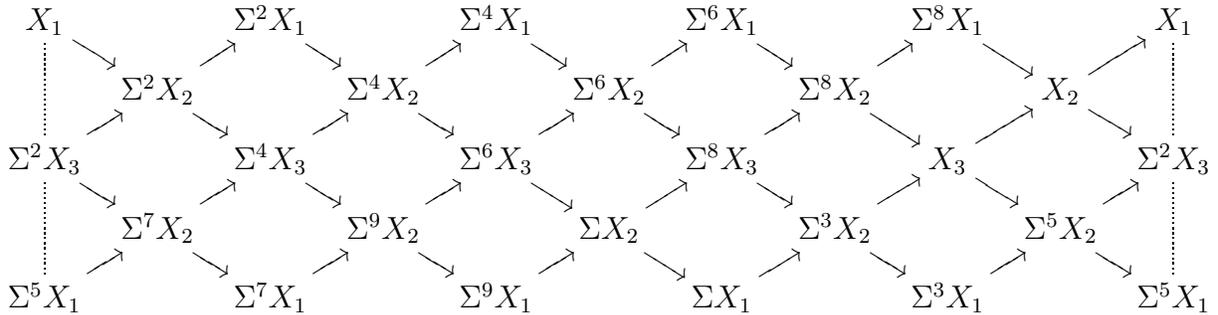
\[ \xymatrix@=2.8mm{
X_1\ar[dr]\ar@{.}[dd]&&\Sigma^2X_1\ar[dr]&&\Sigma^4X_1\ar[dr]&&
\Sigma^6X_1\ar[dr]&&\Sigma^8X_1\ar[dr]&&X_1\ar@{.}[dd]\\
&\Sigma^2X_2\ar[ur]\ar[dr]&&\Sigma^4X_2\ar[dr]\ar[ur]&&
\Sigma^6X_2\ar[dr]\ar[ur]&&\Sigma^8X_2\ar[dr]\ar[ur]&&X_2\ar[dr]\ar[ur]&\\
\Sigma^2X_3\ar[dr]\ar[ur]\ar@{.}[dd]&&\Sigma^4X_3\ar[dr]\ar[ur]&&
\Sigma^6X_3\ar[dr]\ar[ur]&&\Sigma^8X_3\ar[dr]\ar[ur]&&
X_3\ar[dr]\ar[ur]&&\Sigma^2X_3\ar@{.}[dd] \\
&\Sigma^7X_2\ar[dr]\ar[ur]&&\Sigma^9X_2\ar[dr]\ar[ur]&&
\Sigma X_2\ar[dr]\ar[ur]&&\Sigma^3X_2\ar[dr]\ar[ur]&&
\Sigma^5X_2\ar[dr]\ar[ur]& \\
\Sigma^5X_1\ar[ur]&&\Sigma^7 X_1\ar[ur]&&\Sigma^9X_1\ar[ur]&&
\Sigma X_1\ar[ur]&&\Sigma^3X_1\ar[ur]&&\Sigma^5X_1
} \]
Here, the left and right ends are identified along the dotted lines 
to form a cylinder of height $5$ and circumference $5$, for a total of
$25$ isomorphism classes of objects.
Note that in this example we have $X_i\simeq \Sigma^{10}X_i$ for each
$i$, and also for the middle row $X_3\simeq \Sigma^5X_3$. In general,
$\Sigma^b$ acts as a reflection on the quiver, about the horizontal line going through
the middle. There are $[h/2]$ orbits of the shift functor. If $h$ is
odd, they all have 
length $2|b|$, while if $h$ is even there is just one of them with
length $|b|$, namely the middle one, and the rest have length $2|b|$.

\section{Classification of indecomposables}

The goal of this section is to show that every indecomposable object 
in $\Dsg(B)$ is isomorphic to some $\Sigma^j X_i$, so that the
Auslander--Reiten quiver is that given in Theorem~\ref{th:AR-quiver}.

The following proposition plays a role analogous to that of the
Harada--Sai lemma in the representation theory of finite dimensional
algebras, see for example Lemma~4.14.1 of~\cite{Benson:1998c}.

\begin{proposition}\label{pr:zero}
The composite of any $h$ composable irreducible morphisms between the
objects $\Sigma^j X_i$ is equal to zero.
\end{proposition}
\begin{proof}
For each $i$ and $j$, the sum of the composites
$\Sigma^{j-2a}X_i \to \Sigma^jX_{i+1} \to \Sigma^jX_i$ and $\Sigma^{j-2a}X_i \to
\Sigma^{j-2a}X_{i-1}\to \Sigma^jX_i$ is the composite of two adjacent maps in an
Auslander--Reiten triangle, and therefore equal to zero. Using these relations, 
any composite of $h$
composable irreducible morphisms can be rewritten as a composite
involving $\Sigma^{j-2a} X_1 \to \Sigma^j X_2 \to \Sigma^j X_1$, 
which is zero. In other words, we can deform any path of length $h$ so
that it hits the top edge of the cylinder, without moving the ends of
the path.
\end{proof}

\begin{proposition}\label{pr:map}
If $X$ is any non-zero object in $\Dsg(B)$ then for some $n\in
\bZ$ there is a non-zero morphism $ \Sigma^nX_1 \to X$.
\end{proposition}
\begin{proof}
By Corollary~\ref{co:HomDsgB}, 
the image of $X_1=k$ under the equivalence $\Dsg(B)\simeq
\Db(A[\tau^{-1}])$ is $A[\tau^{-1}]$. So if $Y$ is the image of $X$ in
$\Db(A[\tau^{-1}])$ then 
\begin{equation*}
 \Hom_{\Dsg(B)}(X_1,X) \cong \Hom_{\Db(A[\tau^{-1}])}(A[\tau^{-1}],Y) 
\cong H_*Y. 
\end{equation*}
if there is no non-zero morphism from any $\Sigma^n X_1$ to $X$ then
$H_*Y=0$, so $Y$ is quasi-isomorphic to zero in $\Db(A[\tau^{-1}])$
and hence $X$ is quasi-isomorphic to zero in $\Dsg(B)$. 
\end{proof}

\begin{theorem}\label{th:classification}
Every indecomposable object in $\Dsg(B)$ is isomorphic to 
some $\Sigma^j X_i$ with $1\le i< h$,  $0\le j<|b|$.
\end{theorem}
\begin{proof}
First we note that the set of $\Sigma^j X_i$ with $1\le i < h$ and
$0\le j < b$ is the set of vertices in the 
Auslander--Reiten quiver described in Section~\ref{sec:ARtriangles}.

Let $X$ be an indecomposable object. Then by Proposition~\ref{pr:map}
there is a non-zero morphism $ \Sigma^n X_1\to X$ for some $n\in
\bZ$. Since $X$ is indecomposable, if this has a left 
inverse then it is an isomorphism, and we are done. Otherwise, it
factors as $\Sigma^n X_1 \to \Sigma^{n+2a} X_2\to X$. 
If $\Sigma^{n+2a}X_2\to X$ has a left inverse, again we are done. 
Otherwise, we obtain a factorisation 
\[ \Sigma^n X_1 \to \Sigma^{n+2a}X_2 \to \Sigma^{n+2a}X_1\oplus
  \Sigma^{n+4a} X_3 \to X. \]
Since the composite $\Sigma^n X_1 \to \Sigma^{n+2a}X_2\to
\Sigma^{n+2a}X_1$ is zero, it follows that the composite $\Sigma^n X_1
\to \Sigma^{n+2a}X_2 \to \Sigma^{n+4a}X_3 \to X$ is
non-zero. If $\Sigma^{n+4a}X_3\to X$ is not an isomorphism, then at 
the next stage, we obtain a statement that a sum of two
composites is non-zero, so at least one of them has to be non-zero.
Continuing this way, we obtain either an isomorphism
$\Sigma^jX_i\to X$ for some $i,j$, or a factorisation through a
composite of at least $h$ irreducible morphisms between the objects $X_i$.
In the latter case, by
Proposition~\ref{pr:zero}, it follows that $ \Sigma^n X_1\to X$ is the
zero map, contradicting the way it was chosen. 
\end{proof}

\begin{corollary}\label{co:AR-quiver}
The Auslander--Reiten quiver of $\Dsg(B)$ is isomorphic to $\bZ
A_{h-1}/T^{|b|}$.
\end{corollary}
\begin{proof}
This follows from Theorems~\ref{th:AR-quiver} and~\ref{th:classification}.
\end{proof}

\begin{corollary}
The Krull--Schmidt theorem holds in $\Dsg(B)$.
\end{corollary}
\begin{proof}
This follows from the last statement of Theorem~\ref{th:EndDsgXi} together
with Theorem~\ref{th:classification}.
\end{proof}

\begin{remark}
In Remark \ref{rem:formal}  we noted that for the formal graded rings
$H_*(B)$ the objects $X_s$ exist for all $s\ge 1$ and are
inequivalent. The singularity category retains the $\tau$-periodicity,
since $\Dsg (H_*(B))\simeq \Dcsg (H_*(A))\simeq \Db (H_*(A)[1/\tau])$,  
but now (in the absence of Proposition~\ref{pr:X_{h-i}})
 the Auslander--Reiten quiver consists of two semi-infinite cylinders
 of circumference $|b|$, one containing the even suspensions of the
 $X_i$ and one containing the  odd suspensions of the $X_i$. 
\end{remark}

\section{\texorpdfstring{Models for $\Dsg(B)$}{Models for Dsg(B)}}%
\label{se:models}

In this section we exhibit some more familiar looking categories that
are equivalent as triangulated categories to $\Dsg(B)$.

Theorem~\ref{th:classification} gives us a presentation for the category
$\Dsg(B)$ in terms of the \emph{mesh category} of the quiver. The definition of
the mesh category $k(\Gamma)$ of a translation quiver $\Gamma$
comes from Riedtmann~\cite{Riedtmann:1980a} (see also Bongartz and
Gabriel~\cite{Bongartz/Gabriel:1982a}), and can be found in
Section~I.5.6 on pages 54--55 of Happel~\cite{Happel:1988a}. 

In the case of the Auslander-Reiten quiver $\bZ A_{h-1}/T^{|b|}$ of
Corollary~\ref{co:AR-quiver}, 
the generators for the morphisms in the mesh category are the irreducible morphisms between
the $\Sigma^j X_i$. The mesh relations come from the Auslander--Reiten
triangles, and say that for each $i$ and $j$, the sum of the composites
$\Sigma^{j-2a}X_i \to \Sigma^jX_{i+1} \to \Sigma^jX_i$ and 
$\Sigma^{j-2a}X_i \to \Sigma^{j-2a}X_{i-1}\to \Sigma^jX_i$ is 
equal to zero. At the boundaries, $i=1$
and $i=h-1$, only one of these composites makes sense, and the
corresponding relation says that this composite is equal to zero.

The classification of Krull--Schmidt triangulated
categories with finitely many isomorphism classes of indecomposables is
described in Section~6 of Amiot~\cite{Amiot:2007a}, see also Chapter~2 
of Amiot~\cite{Amiot:2008a} and the paper of 
Xiao and Zhu~\cite{Xiao/Zhu:2002a}. Let $\Gamma$ be the translation
quiver $\bZ A_{h-1}/T^{|b|}$. Then it is shown in Theorem~6.5
of~\cite{Amiot:2007a} that given 
a triangulated category $\cT$ with Auslander--Reiten quiver $\Gamma$, 
we have an equivalence of $k$-linear categories between the full subcategory $\ind(\cT)$ of
indecomposables in $\cT$ and the mesh category $k(\Gamma)$. This
induces an equivalence between $\cT$ and
the additive closure of $k(\Gamma)$.
Applying this to $\Dsg(B)$, we obtain a $k$-linear equivalence 
$k(\bZ A_{h-1}/T^{|b|})\to\ind\Dsg(B)$, the full subcategory of
indecomposables, which then extends to an equivalence
from the additive closure of $k(\bZ A_{h-1}/T^{|b|})$ to $\Dsg(B)$.

There is another triangulated category with the same Auslander--Reiten
quiver. Let $\Db(kA_{h-1})$ be the bounded derived category of modules
for the quiver $A_{h-1}$ over $k$. The Auslander--Reiten quiver of
$\Db(kA_{h-1})$ is the quiver $\bZ A_{h-1}$. The translation $T$ of
this quiver lifts to the translation $\sfT$ of $\Db(kA_{h-1})$.
It is shown in
Keller~\cite{Keller:2005a} that the \emph{orbit category} 
$\Db(kA_{h-1})/\sfT^{|b|}$,
whose Hom sets are by definition
\[ \bigoplus_{n\in\bZ}\Hom_{\Db(kA_{h-1})}(X,\sfT^{n|b|}(Y)), \] 
is a triangulated category in such a way that the canonical functor
\[ \Db(kA_{h-1}) \to \Db(kA_{h-1})/\sfT^{|b|} \] 
is a triangle functor. The Auslander--Reiten quiver of the orbit
category $\Db(kA_{h-1})/\sfT^{|b|}$ is isomorphic to $\bZ A_{h-1}/T^{|b|}$.
So Theorem~6.5 of~\cite{Amiot:2007a} shows that there is a $k$-linear
equivalence 
\begin{equation*}
\Dsg(B) \simeq \Db(kA_{h-1})/\sfT^{|b|}
\end{equation*} 
inducing the isomorphism of Auslander--Reiten quivers. We would like
to know that this is an equivalence of triangulated categories. This
proves to be more delicate, but another theorem of Amiot comes to our rescue.

\begin{theorem}\label{th:DbkAh-1/T^b}
There is a triangulated equivalence $\Dsg(B)\simeq \Db(kA_{h-1})/\sfT^{|b|}$.
\end{theorem}
\begin{proof}
We would like to apply 
Amiot~\cite[Theorem~7.2]{Amiot:2007a}. This states 
that given a finite triangulated category $\cT$ which is connected,
algebraic, and standard, there exists a Dynkin diagram $\Gamma$ of
type $A$, $D$ or $E$, and a self-equivalence $\Phi$ of $\Db(k\Gamma)$,
such that $\cT$ is triangle equivalent to Keller's orbit category
$\Db(k\Gamma)/\Phi$. To apply the theorem, we need to check the
conditions. 

To say that $\cT$ is connected means that the Auslander--Reiten quiver
is connected, so we have already established that $\Dsg(B)$ is connected.
To say that $\cT$ is standard means that it is equivalent to a mesh 
category as a $k$-linear category, so we have also already established that 
$\Dsg(B)$ is standard. 

To say that $\cT$ is algebraic means that there is a
Frobenius category $\cE$ such that $\cT$ is triangle equivalent to the
stable category $\underline\cE$, see 
Keller~\cite[Section~3.6]{Keller:2006a}. 
By Proposition~\ref{pr:DbQsimDbAtau^-1}
and Corollary~\ref{co:DsgBsimDbAtau-1} we have
$\Db(Q[\tau^{-1}])\simeq\Db(A[\tau^{-1}])\simeq\Dsg(B)$, and 
$\Db(Q[\tau^{-1}])$ is algebraic 
by~\cite[Lemma~3.3 and Theorem~3.9]{Keller:2006a}.
It follows that $\Dsg(B)$ is algebraic.

We have therefore checked the conditions for applying the theorem of
Amiot. Since the Auslander--Reiten quiver of $\Dsg(B)$ is $\bZ
A_{h-1}/T^{|b|}$, we have $\Gamma=A_{h-1}$ and $\Phi=\sfT^{|b|}$, and the
theorem now follows.
\end{proof}

There is another category that looks very similar, and we apply
similar techniques to make the comparison.
We write $\bar B$ for the formal $A_\infty$ algebra $k[x]/(x^h)$ where
$|x|=-2a$. Thus there is an obvious map $B \to \bar B$ sending $x$ to
$x$ and $t$ to zero.
We consider the bounded derived category $\Db(\bar B)$ 
and its quotient, the singularity category $\Dsg(\bar B)$ formed
by quotienting out all objects finitely built from the ring. We have
objects $\bar M_i=k[x]/(x^i)$ in this category for $1\le i\le h$, and 
$\bar M_h$ is zero. The analogue of 
Corollary~\ref{co:Aduality} in this situation says that 
$\Hom_{\Dsg(\bar B)}(X,Y)$ is the graded vector space dual of
$\Hom_{\Dsg(\bar B)}(Y,\Sigma^{1-2a}X)$, and hence we have
Auslander--Reiten triangles 
\[ \Sigma^{-2a}\bar M_i \to
\Sigma^{-2a}\bar M_{i-1}\oplus \bar M_{i+1} \to \bar M_i. \]
However, in contrast to the situation for $\Dsg(B)$, the
periodicity is given by 
\[ \bar M_i\cong \Sigma^{2\ell b}\bar M_i, \]
since $\Omega^2\bar M_i\cong \Sigma^{2ha}\bar M_i$, and $2ha-2=2\ell b$.  
The Auslander--Reiten quiver again consists of the $\Sigma^j\bar
M_i$. It is in the form of a cylinder, and again the
height of the cylinder is $h-1$, but the circumference is $\ell |b|$
instead of $|b|$. The generators and relations for this category are
given in the same way as that of $\Dsg(B)$ in terms of the
irreducible morphisms and the Auslander--Reiten triangles.

\begin{theorem}
There is a triangulated equivalence $\Dsg(\bar B)\simeq
\Db(kA_{h-1})/\sfT^{\ell|b|}$.
\end{theorem}
\begin{proof}
This is proved in the same way as Theorem~\ref{th:DbkAh-1/T^b}, using 
Amiot~\cite[Theorem~7.2]{Amiot:2007a}.
\end{proof}

The functor $\Dsg(\bar B) \to \Dsg(B)$ sends $\bar M_i$ to $\Sigma X_i$
and irreducible morphisms to irreducible morphisms. It wraps 
the Auslander--Reiten quiver of $\Dsg(\bar B)$ around
that of $\Dsg(B)$ exactly $\ell$ times. Thus it corresponds to the
functor on orbit categories
\[ \Db(kA_{h-1})/\sfT^{\ell|b|} \to \Db(kA_{h-1})/\sfT^{|b|}. \]

There is another way to achieve this wrapping around. Namely, instead
of considering differential $\bZ$-graded modules for $\bar B$, we
consider differential $\bZ/2|b|$-graded modules. Let us write
$\Db(\bar B,\bZ/2|b|)$ for this bounded derived category
$\Dsg(\bar B,{\bZ/2b})$ for the corresponding singularity 
category. 

\begin{theorem}
There is a triangulated equivalence $\Dsg(\bar B,\bZ/2|b|)\simeq 
\Db(kA_{h-1})/\sfT^{|b|}$.
There is an
equivalence of categories $\Dsg(\bar B,{\bZ/2|b|})\to
\Db(B)$ making the following diagram commute.
\[ \xymatrix{\Dsg(\bar B) \ar[r] \ar[d] & \Dsg(B) \\
\Dsg(\bar B,{\bZ/2|b|}) \ar[ur]_\simeq} \]
\end{theorem}
\begin{proof}
Again, this follows by applying Amiot~\cite[Theorem~7.2]{Amiot:2007a}.
\end{proof}

We put all these equivalences together in the following theorem.

\begin{theorem}\label{th:frt}
We have equivalences of triangulated categories
\[ \Db(kA_{h-1})/\sfT^{|b|}\simeq \Dsg(\bar B,\bZ/2|b|) \simeq
  \Dsg(B)\simeq \Dcsg(A)\simeq \Db(A[\tau^{-1}])\simeq \Db(Q[\tau^{-1}]). \]
Each of these is a finite Krull--Schmidt category with $|b|(h-1)$ isomorphism classes of indecomposable
objects, in $[h/2]$ orbits of the shift functor. The Auslander--Reiten
quiver is $\bZ A_{h-1}/T^{|b|}$.\qed
\end{theorem}

\section{\texorpdfstring{$H^*BG$ and $H_*\Omega BG\phat$}
{H*BG and HᕯΩBGphat}}\label{se:BG}

In this section, we apply our main results to the $A_\infty$ algebras
$H^*BG$ and $H_*\Omega BG\phat$ for $G$ a finite group with cyclic
Sylow $p$-subgroups, discussed in~\cite{Benson/Greenlees:bg8}.

We are interested in the following occurrences of the $A_\infty$
algebra $A$ of Section~\ref{se:A}.
Let $p$ be an odd prime and $k$ a field of characteristic $p$.
Let $G$ be a finite group with cyclic Sylow subgroup $P$ of order
$p^n$ and inertial index $q=|N_G(P){:}C_G(P)|>1$. 
Let $C^*BG$ be the
cochains on the classifying space $BG$ with coefficients in $k$, and
$C_*\Omega BG\phat$ be the chains on  the $p$-completed loop space of
$BG$, again with coefficients in $k$.

The $A_\infty$
algebra structures on $H^*BG$ and on $H_*\Omega BG\phat$ coming from
the DG algebras $C^*BG$ and $C_*\Omega BG\phat$ are described
in~\cite{Benson/Greenlees:bg8}. They are Koszul dual, and we can apply 
the results of this paper either with $A=H^*BG$, $B=H_*\Omega
BG\phat$, or with $A=H_*\Omega BG\phat$ and $B=H^*BG$.

In the case of $A=H_*\Omega BG\phat$, we have $a=q$, $b=q-1$,
$h=p^n-(p^n-1)/q$ and $\ell=p^n$, and we must assume that
$q>1$.
In the case of $A=H^*BG$ (homologically graded in negative degrees), 
the roles are reversed and 
we have $a=-(q-1)$, $b=-q$,  $h=p^n$, and $\ell=p^n-(p^n-1)/q$.

\begin{proof}[Proof of Theorem~\ref{th:main}]
By Proposition~\ref{pr:DbQsimDbA}, the bounded derived categories of the
DGA algebras $C^*BG$ and $C_*\Omega BG\phat$ are equivalent to those
of the $A_\infty$ algebras $H^*BG$ and $H_*\Omega BG\phat$. 
The equivalences of categories follow by applying
Theorem~\ref{th:DbAsimDbB} to the 
$A_\infty$ algebras
$H^*BG$ and $H_*\Omega BG\phat$. The classification of the
indecomposable objects in these categories follow from Theorem~\ref{th:frt}.
\end{proof}

\section{Brauer trees and Hecke algebras}

In this section, we describe what our main theorem tells us about
Brauer tree algebras. 
In general, a Brauer tree algebra is described by a planar embedding
of a tree with $e$ edges corresponding to the simple modules. The
vertices are assigned multiplicities, which are all equal to one with
the possible exception of a single vertex of multiplicity
$\lambda>1$; otherwise we set $\lambda=1$. This parameter $\lambda$ is
called the \emph{exceptional multiplicity}, even when it equals one.
These data are sufficient to describe the algebra up to Morita
equivalence, and an algorithm for computing projective resolutions was
described by Green~\cite{Green:1974a}. Brauer tree algebras were first
introduced in order to describe blocks of defect one in the
representation theory of finite groups by Brauer~\cite{Brauer:1941a},
and the analysis was extended to all blocks of cyclic defect by
Dade~\cite{Dade:1966a}. A nice treatment in this context may be found
in the book of Alperin~\cite{Alperin:1986a}.
They also appear in many other contexts in representation theory, 
and we shall give an example of this in characteristic zero below.

We say that a simple module is a \emph{leaf module} if it
corresponds to an edge one end of which has valency one and
multiplicity one. The leaf modules are all syzygies of each other,
so they have isomorphic Ext algebras.

\begin{theorem}\label{th:Brauer-tree}
Let $\mathbf A$ be a Brauer tree algebra with $e>1$ edges and
exceptional multiplicity $\lambda$, and let $M$ be a simple leaf
module for $\mathbf A$.  Then the $A_\infty$ algebra 
$\Ext^*_{\mathbf A}(M,M)$ is the algebra $B$ described in
Section~\ref{se:A}, with parameters $a=e$, $b=e-1$, $\ell=\lambda e +
1$ and $h=\ell-\lambda=\lambda(e-1)+1$.

The singularity category of $\Ext^*_{\mathbf A}(M,M)$ has finite
representation type, with $\lambda(e-1)^2$ isomorphism classes of
indecomposable objects, in $[(\lambda(e-1)+1)/2]$ orbits of the shift functor
$\Sigma$. The Auslander--Reiten quiver is isomorphic to $\bZ
A_{\lambda(e-1)}/T^{e-1}$, where $T$ is the translation functor
$\Sigma^{-2e}$. 

The cosingularity category of $\Ext^*_{\mathbf A}(M,M)$ has finite
representation type, with $\lambda e^2$ isomorphism classes of
indecomposable objects, in $[(\lambda e +1)/2]$ orbits of the shift
functor $\Sigma$. The Auslander--Reiten quiver is isomorphic to $\bZ
A_{\lambda e}/T^e$, where $T$ is the translation functor $\Sigma^{-2(e-1)}$.
\end{theorem}
\begin{proof}
It was shown by Gabriel and Riedtmann~\cite{Gabriel/Riedtmann:1979a}
that all Brauer tree algebras with $e$
edges and exceptional multiplicity $\lambda$ are stably equivalent
(indeed, even more is true:
Rickard~\cite{Rickard:1989a},
see also Aihara~\cite{Aihara:2014a}, showed that these are all derived
equivalent, but we don't need to go that far).
In particular, they are all stably equivalent to the \emph{Brauer star
  algebra} which has $e$ vertices of valency one and multiplicity
one, surrounding the one remaining vertex in the middle, which has
multiplicity $\lambda$. The Brauer star algebras have exactly the same
presentation as the algebras in
Section~2 of~\cite{Benson/Greenlees:bg8}, except for the change of
characteristic. In more detail,  the projective
indecomposables are uniserial, and the radical filtration is isomorphic to its
associated graded. Some readers may find it helpful to refer to the
paper  Bogdanic~\cite{Bogdanic:2010a} which recalls details of Brauer tree
algebras and shows that the stable grading comes from a grading
of  the Brauer tree algebras themselves.

Under such a stable equivalence, the
leaf modules correspond to simple modules or first syzygies of simple
modules, for the corresponding Brauer
star algebra. It follows that we may compute the $A_\infty$ structure
on these Ext algebras using exactly the same computation as 
in~\cite{Benson/Greenlees:bg8}, and the result is as stated in the
theorem.

Theorem~\ref{th:main} now describes the singularity and cosingularity categories.
\end{proof}

\begin{remark}
In the remaining case $e=1$, a Brauer tree algebra is Morita equivalent to a truncated
polynomial algebra, so the $A_\infty$ structure on the Ext ring of the
simple module is just like that of a cyclic $p$-group.
\end{remark}

As an example, let $\cH=\cH(n,q)$ be the Hecke algebra of the symmetric group of
degree $n$ over a field $k$ of characteristic zero, where $q$ is a
primitive $\ell$th root of unity with $\ell\ge 2$. This has generators
$T_1,\dots,T_{n-1}$ satisfying braid relations together with the
relations $(T_i+1)(T_i-q)=0$.
The cohomology $H^*(\cH,k)=\Ext^*_{\cH}(k,k)$ was computed by 
Benson, Erdmann and Mikaelian~\cite{Benson/Erdmann/Mikaelian:2010a}.

In the case where $n=\ell > 2$, $H^*(\cH,k)$ has the form
$k[x] \otimes \Lambda(t)$ where $|x|=-2n+2$ and $|t|=-2n+3$. 
This corresponds to the fact that in this case the principal block of 
$\cH$ is a Brauer tree
algebra for a tree which is a straight line with $n$ vertices, $n-1$ edges, and 
$\lambda=1$, with the trivial module $k$ at one end as a leaf module.
In particular, we can apply Theorem~\ref{th:Brauer-tree} in this context.
So the Massey product of $n$ copies of $t$ is equal to $-x^{n-1}$,
and writing $a=n-1$, $b=n-2$, $\ell=n$, $h=n-1$, we have 
$\Ext^*_{\cH}(k,k)=B$ as an $A_\infty$ algebra. 

\begin{theorem}\label{th:Hecke}
Let $\cH=\cH(n,q)$ be the Hecke algebra of the symmetric group of
degree $n$ over a field $k$ of characteristic zero, where $q$ is a
primitive $\ell$th root of unity.
In the case $n=\ell>2$, the singularity category of the $A_\infty$
algebra $\Ext^*_{\cH}(k,k)$ has finite representation type, with
$(n-2)^2$ isomorphism classes of indecomposables, in $[(n-1)/2]$ orbits of
the shift functor. The
Auslander--Reiten quiver is a cylinder of height $n-2$ and
circumference $n-2$.

The cosingularity category of the $A_\infty$ algebra $\Ext^*_{\cH}(k,k)$ also has finite
representation type, with $(n-1)^2$ isomorphism classes of
indecomposables, in $[n/2]$ orbits of the shift functor. The
Auslander--Reiten quiver is a cylinder of height $n-1$ and
circumference $n-1$.
\end{theorem}

Other examples of Hecke algebras described by Brauer trees may be 
found in Geck~\cite{Geck:1992a}, Ariki~\cite{Ariki:2017a}.

\bibliographystyle{amsplain}
\bibliography{../repcoh}

\newcommand{\noopsort}[1]{}
\providecommand{\bysame}{\leavevmode\hbox to3em{\hrulefill}\thinspace}
\providecommand{\MR}{\relax\ifhmode\unskip\space\fi MR }
\providecommand{\MRhref}[2]{%
  \href{http://www.ams.org/mathscinet-getitem?mr=#1}{#2}
}
\providecommand{\href}[2]{#2}
\begin{thebibliography}{10}

\bibitem{Adams:1969a}
J.~F. Adams, \emph{{Lectures on generalised cohomology}}, {Category Theory,
  Homology Theory and their Applications, III}, Lecture Notes in Mathematics,
  vol.~99, Springer-Verlag, Ber\-lin/New York, 1969, pp.~1--138.

\bibitem{Aihara:2014a}
T.~Aihara, \emph{{Mutating Brauer trees}}, Math.\ J.\ Okayama Univ. \textbf{56}
  (2014), 1--16.

\bibitem{Alperin:1986a}
J.~L. Alperin, \emph{{Local representation theory}}, Cambridge Studies in
  Advanced Mathematics, vol.~11, Cambridge University Press, 1986.

\bibitem{Amiot:2007a}
C.~Amiot, \emph{{On the structure of triangulated categories with finitely many
  indecomposables}}, Bull.\ Soc.\ Math.\ France \textbf{135} (2007), no.~3,
  435--474.

\bibitem{Amiot:2008a}
\bysame, \emph{{Sur les petites cat\'egories triangul\'ees}}, {Th\`ese de
  Doctorat}, Universit\'e Paris VII, 2008.

\bibitem{Ariki:2017a}
S.~Ariki, \emph{{Representation type for block algebras of Hecke algebras of
  classical type}}, Adv.\ in Math. \textbf{317} (2017), 823--845.

\bibitem{Avramov/Foxby/Halperin:dgha}
L.~L. Avramov, H.-B. Foxby, and S.~Halperin, \emph{{Differential graded
  homological algebra}}, Preprint, 2003.

\bibitem{Benson:1998c}
D.~J. Benson, \emph{{Representations and Cohomology II: Cohomology of groups
  and modules}}, Cambridge Studies in Advanced Mathematics, vol. 31, Second
  Edition, Cambridge University Press, 1998.

\bibitem{Benson/Erdmann/Mikaelian:2010a}
D.~J. Benson, K.~Erdmann, and A.~Mikaelian, \emph{{Cohomology of Hecke
  algebras}}, Homology, Homotopy \& Appl. \textbf{12} (2010), no.~2, 353--370.

\bibitem{Benson/Greenlees:bg8}
D.~J. Benson and J.~P.~C. Greenlees, \emph{{Massey products in the homology of
  the loopspace of a $p$-completed classifying space: finite groups with cyclic
  Sylow $p$-subgroups}}, Preprint, 2020.

\bibitem{Boardman:1999a}
J.~M. Boardman, \emph{{Conditionally convergent specral sequences}}, Homotopy
  invariant algebraic strucurs (J.-P. Meyer, J.~Morava, and W.~S. Wilson,
  eds.), Contemp.\ Math., vol. 239, American Math.\ Society, 1999, pp.~49--84.

\bibitem{Bogdanic:2010a}
D.~Bogdanic, \emph{{Graded Brauer tree algebras}}, J.~Pure \& Applied Algebra
  \textbf{214} (2010), 1534--1552.

\bibitem{Bongartz/Gabriel:1982a}
K.~Bongartz and P.~Gabriel, \emph{{Covering spaces in representation-theory}},
  Invent.\ Math. \textbf{65} (1982), no.~3, 331--378.

\bibitem{Brauer:1941a}
R.~Brauer, \emph{{Investigations on group characters}}, Ann.\ of Math.
  \textbf{42} (1941), no.~4, 936--958.

\bibitem{Buchweitz/Roberts:2015a}
R.-O. Buchweitz and C.~Roberts, \emph{{The multiplicative structure on
  Hochschild cohomology of a complete intersection}}, J.~Pure \& Applied
  Algebra \textbf{219} (2015), 402--428.

\bibitem{Buijs/MorenoFernandez/Murillo:2020a}
U.~Buijs, J.~M. Moreno-Fern\'andez, and A.~Murillo, \emph{{$A_\infty$
  structures and Massey products}}, Mediterr.\ J.\ Math. \textbf{17} (2020),
  no.~1, Paper No. 31, 15pp.

\bibitem{Dade:1966a}
E.~C. Dade, \emph{{Blocks with cyclic defect groups}}, Ann.\ of Math.
  \textbf{84} (1966), 20--48.

\bibitem{Ene/Popescu:2008a}
V.~Ene and D.~Popescu, \emph{{On the structure of maximal Cohen-Macaulay
  modules over the ring $K[[x,y]]/(x^n)$}}, Algebras and Representation Theory
  \textbf{11} (2008), 191--205.

\bibitem{Gabriel/Riedtmann:1979a}
P.~Gabriel and Ch. Riedtmann, \emph{{Group representations without groups}},
  Comment.\ Math.\ Helvetici \textbf{54} (1979), 240--287.

\bibitem{Geck:1992a}
M.~Geck, \emph{{Brauer trees of Hecke algebras}}, Comm.\ Algebra \textbf{20}
  (1992), no.~10, 2937--2973.

\bibitem{Getzler/Jones:1990a}
E.~Getzler and J.~D.~S. Jones, \emph{{$A_\infty$-algebras and the cyclic bar
  complex}}, Illinois J.\ Math. \textbf{34} (1990), no.~2, 256--283.

\bibitem{Green:1974a}
J.~A. Green, \emph{{Walking around the Brauer tree}}, J.~Austral.\ Math.\ Soc.
  \textbf{17} (1974), 197--213.

\bibitem{Greenlees/May:1995a}
J.~P.~C. Greenlees and J.~P. May, \emph{{Generalized Tate cohomology}}, vol.
  113, Mem.\ AMS, no. 543, American Math.\ Society, 1995.

\bibitem{Greenlees/Stevenson:2020a}
J.~P.~C. Greenlees and G.~Stevenson, \emph{{Morita theory and singularity
  categories}}, Adv.\ in Math. \textbf{365} (2020), 107055, 51pp.

\bibitem{Greenlees/Stojanoska:2018a}
J.~P.~C. Greenlees and V.~Stojanoska, \emph{{Anderson and Gorenstein duality}},
  Geometric and topological aspects of the representation theory of finite
  groups (J.~F. Carlson, S.~B. Iyengar, and J.~Pevtsova, eds.), Springer Proc.
  Math. Stat., vol. 242, Springer-Verlag, Ber\-lin/New York, 2018,
  pp.~105--130.

\bibitem{Happel:1987a}
D.~Happel, \emph{{On the derived category of a finite-dimensional algebra}},
  Comment.\ Math.\ Helvetici \textbf{62} (1987), 339--389.

\bibitem{Happel:1988a}
\bysame, \emph{{Triangulated categories in the representation theory of finite
  dimensional algebras}}, London Math.\ Soc.\ Lecture Note Series, vol. 119,
  Cambridge University Press, 1988.

\bibitem{Hovey:1999a}
M.~Hovey, \emph{{Model categories}}, Mathematical Surveys and Monographs,
  vol.~63, American Math.\ Society, 1999.

\bibitem{Kadeishvili:1982a}
T.~V. Kadeishvili, \emph{{The algebraic structure in the homology of an
  $A(\infty)$-algebra}}, Soobshch.\ Akad.\ Nauk Gruzin.\ SSR \textbf{108}
  (1982), 249--252.

\bibitem{Keller:dih}
B.~Keller, \emph{{Derived invariance of higher structures on the Hochschild
  complex}}, Preprint, 2018.

\bibitem{Keller:2001a}
\bysame, \emph{{Introduction to $A$-infinity algebras and modules}}, Homology,
  Homotopy \& Appl. \textbf{3} (2001), 1--35, --- Addendum, ibid. {\bf 4}
  (2002), 25--28.

\bibitem{Keller:2002a}
\bysame, \emph{{$A$-infinity algebras in representation theory}},
  Representations of Algebras, Proceedings of the Ninth International
  Conference (Beijing 2000) (D.~Happel and Y.~B. Zhang, eds.), Beijing Normal
  University Press, 2002, Vol. I.

\bibitem{Keller:2005a}
\bysame, \emph{{On triangulated orbit categories}}, Doc.\ Math. \textbf{10}
  (2005), 551--581.

\bibitem{Keller:2006b}
\bysame, \emph{{A-infinity algebras, modules and functor categories}}, Trends
  in Representation Theory of Algebras and Related Topics (J.~A. de~la Pe\~na
  and R.~Bautista, eds.), Contemp.\ Math., vol. 406, 2006, pp.~67--93.

\bibitem{Keller:2006a}
\bysame, \emph{{On differential graded categories}}, International Congress of
  Mathematicians. Vol. II, Eur. Math. Soc., Z\"urich, 2006, pp.~151--190.

\bibitem{Keller:2021a}
\bysame, \emph{{A remark on Hochschild cohomology and Koszul duality}},
  Advances in the Representation Theory of Algebras (I.~Assem, C.~Gei\ss\, and
  S.~Trepode, eds.), Contemp.\ Math., vol. 761, 2021, pp.~131--136.

\bibitem{Lefevre-Hasegawa:2003a}
K.~Lef\`evre-Hasegawa, \emph{{Sur les $A_\infty$-cat\'egories}}, Th\`ese de
  doctorat, Universit\'e Paris VII, 2003.

\bibitem{Lu/Palmieri/Wu/Zhang:2009a}
D.-M. Lu, J.~H. Palmieri, Q.-S. Wu, and J.~J. Zhang, \emph{{$A$-infinity
  structure on Ext-algebras}}, J.~Pure \& Applied Algebra \textbf{213} (2009),
  no.~11, 2017--2037.

\bibitem{Rickard:1989a}
J.~Rickard, \emph{{Derived categories and stable equivalence}}, J.~Pure \&
  Applied Algebra \textbf{61} (1989), 303--317.

\bibitem{Riedtmann:1980a}
C.~Riedtmann, \emph{{Algebren, Darstellungsk\"ocher, Ueberlagerungen und
  zur\"uck}}, Comment.\ Math.\ Helvetici \textbf{55} (1980), 199--224.

\bibitem{Roitzheim/Whitehouse:2011a}
C.~Roitzheim and S.~Whitehouse, \emph{{Uniqueness of $A_\infty$-structures and
  Hochschild cohomology}}, Algebr.\ Geom.\ Topol. \textbf{11} (2011), 107--143.

\bibitem{Stasheff:1970a}
J.~D. Stasheff, \emph{{$H$-spaces from a homotopy point of view}}, Lecture
  Notes in Mathematics, vol. 161, Springer-Verlag, Ber\-lin/New York, 1970.

\bibitem{Xiao/Zhu:2002a}
J.~Xiao and B.~Zhu, \emph{{Relations for the Grothendieck groups of
  triangulated categories}}, J.~Algebra \textbf{257} (2002), 37--50.

\end{thebibliography}

\end{document}